\documentclass[12pt,a4paper]{article}
\usepackage{amssymb,amsmath,amsthm}
\usepackage[english]{babel}
\usepackage{t1enc}
\usepackage[latin2]{inputenc}
\usepackage{epsfig}
\usepackage{subfig}
\usepackage{epstopdf}
\usepackage{xcolor}
\usepackage{comment}
\usepackage{hyperref}

\newtheorem{thm}{Theorem}
\newtheorem{cor}[thm]{Corollary}
\newtheorem{defi}[thm]{Definition}
\newtheorem{lem}[thm]{Lemma}

\newtheorem{claim}[thm]{Claim}

\newtheorem{prob}[thm]{Problem}
\newtheorem{obs}[thm]{Observation}

\theoremstyle{remark}

\usepackage{indentfirst}
\def\HH{\mathcal {X}}

\def\red1{\color{red}1\color{black}}
\def\blue1{\color{blue}1\color{black}}

\def\ca{\circlearrowright}
\def\01{0\text{-}1}

\makeatletter
\def\blfootnote{\gdef\@thefnmark{}\@footnotetext}
\makeatother

\begin{document}

\title{Saturation of Ordered Graphs}
\author{Vladimir Bo\v skovi\'c\thanks{Sorbonne Universit\'e, Paris Graduate School of Mathematical Sciences funded by Foundation Sciences Math\'ematiques de Paris, France.} 
	\and Bal\'azs Keszegh\thanks{Alfréd Rényi Institute of Mathematics and Eötvös Loránd University, MTA-ELTE Lendület Combinatorial Geometry Research Group, Budapest, Hungary. Research supported by the Lend\"ulet program of the Hungarian Academy of Sciences (MTA), under the grant LP2017-19/2017, by the J\'anos Bolyai Research Scholarship of the Hungarian Academy of Sciences, by the National Research, Development and Innovation Office -- NKFIH under the grant K 132696 and FK 132060 and by the ÚNKP-20-5 New National Excellence Program of the Ministry for Innovation and Technology from the source of the National Research, Development and Innovation Fund.}}
\maketitle


\begin{abstract}

Recently, the saturation problem of $0$-$1$ matrices gained a lot of attention. This problem can be regarded as a saturation problem of ordered bipartite graphs. Motivated by this, we initiate the study of the saturation problem of ordered and cyclically ordered graphs. 

We prove that dichotomy holds also in these two cases, i.e., for a (cyclically) ordered graph its saturation function is either bounded or linear. We also determine the order of magnitude for large classes of (cyclically) ordered graphs, giving infinite many examples exhibiting both possible behaviours, answering a problem of Pálvölgyi. In particular, in the ordered case we define a natural subclass of ordered matchings, the class of linked matchings, and we start their systematic study, concentrating on linked matchings with at most three links and prove that many of them have bounded saturation function.

In both the ordered and cyclically ordered case we also consider the semisaturation problem, where dichotomy holds as well and we can even fully characterize the graphs that have bounded semisaturation function.
\end{abstract}

\section{Introduction}
Extremal problems are among the most studied topics in combinatorics. We are usually interested in determining or estimating the maximum density (this is called the extremal function) of combinatorial objects avoiding a given substructure that we call the \emph{forbidden} substructure. The most important case is the case of graphs where density means the number of edges. The earliest example is the classic problem of determining the maximum number of edges in a graph without a triangle by Mantel or more generally without a complete graph on $k$ vertices by Tur\'an. Since then extremal problems are very popular and have a huge literature. Here we have no space to discuss these results in more detail, in the rest we only concentrate on the saturation problem.

In a \emph{saturation} problem, instead of the maximum density of an avoiding structure, we are interested in estimating the minimum density of an avoiding structure such that extending it in an arbitrary way it introduces the forbidden substructure. Another common variant of saturation problems is the so-called \emph{semisaturation}\footnote{Semisaturation is also sometimes called strong saturation or oversaturation.} problem \cite{furedikim}, in which we are interested in the smallest density of a (not necessarily avoiding) structure such that extending it in an arbitrary way it introduces a new copy of the forbidden substructure. In case of graphs, e.g., a saturation problem is to determine the minimum number of edges of a graph which does not contain a triangle but adding an arbitrary edge introduces a triangle. Further, it is a semisaturation problem to determine the minimum number of edges of a graph (which may contain a triangle) such that adding an arbitrary edge creates a new triangle. Erd\H os, Hajnal and Moon were the first to investigate such functions \cite{erdoshajnalmoon}, determining the minimal numbers when the forbidden structure is a complete graph on $k$ vertices. For a survey of saturation problems of graphs see \cite{graphsatsurvey}.

Besides graphs there are many further settings where in addition to the extremal problems the saturation problems are also investigated, such as set systems and $0$-$1$ matrices among others \cite{2020saturation,dudek2013minimum,ferrara2017saturation,frankl2020vc,Gerbner2013,keszegh2020induced}. The case of $0$-$1$ matrices leads us to the case of ordered graphs, our aim is to investigate their saturation functions.

\subsection{Saturation functions}

Before proceeding further we define the notions that interest us. 
An \emph{ordered (resp. cyclically ordered) graph} is a graph whose vertex set is linearly (resp. cyclically) ordered. A \emph{matching} is a graph in which every vertex has a degree one. The \emph{interval chromatic number} of a (cyclically) ordered graph $G$ is the minimum number of intervals the (cyclically) ordered vertex set of $G$ can be partitioned into so that no two vertices belonging to the same interval are adjacent in $H$. The graphs we deal with are simple, i.e., have no parallel nor loop edges. When we say that we add an edge $e$ to some graph $G$ we always assume that $e$ is not an edge of $G$.

For a non-empty graph $G$ let $sat(n,G)$ be the minimum number of edges in a graph $H$ on $n$ vertices, with the property that it does not contain $G$ as a subgraph but adding an arbitrary edge to $H$ creates a copy of $G$. For an ordered (resp. cyclically ordered) non-empty graph $G$ let $sat_<(n,G)$ (resp. $sat_\ca(n,G)$) be the number of edges in an ordered (resp. cyclically ordered) graph $H$ on vertex set of size $n$ and minimum number of edges, with the property that it does not contain $G$ as an ordered (resp. cyclically ordered) subgraph but adding an arbitrary edge to $H$ creates a copy of $G$. The case of $0$-$1$ matrices is also relevant for us, we define it equivalently on ordered bipartite graphs. For an ordered non-empty graph $G$ with interval chromatic number $2$ let $sat_{\01}(n,G)$ be the number of edges in an ordered bipartite graph $H$ on $n+n$ vertices with all edges between the first $n$ and the last $n$ vertices\footnote{i.e., $H$ has interval chromatic number $2$ and can be split into two intervals of size $n$.}
and with minimum number of edges, with the property that it does not contain $G$ as an ordered subgraph but adding an arbitrary edge to $H$ between its two parts creates a copy of $G$\footnote{By taking the incidence matrices of $G$ and $H$ we get the equivalent setting of forbidding a $0$-$1$ matrix in an $n$ by $n$ $0$-$1$ matrix. This setting is the one used usually when considering this problem.}. Note that we may have defined similarly the bipartite version of the unordered graph saturation, but as we won't discuss it in detail, this is omitted.

We say that $H$ is \emph{saturating} $G$ if $H$ avoids $G$ but adding an arbitrary edge introduces a copy of $G$. In all cases we refer to $G$ as the \emph{forbidden} graph and $H$ as the minimal saturated \emph{host} graph. 

When we do not require that the host graph $H$ avoids $G$, instead we only require that adding a new edge to $H$ creates a new copy of $G$, we get the semisaturation problem. The number of edges in minimal semisaturated graphs in the above settings give the functions $ssat(n,G)$, $ssat_<(n,G)$, $ssat_\ca(n,G)$ and $ssat_{\01}(n,G)$.

As a warm-up to these notions, notice that for the complete graph $K_k$ on $k$ vertices by definition $sat(n,K_k)=sat_<(n,K_k)=sat_\ca(n,K_k)$ and similarly the semisaturation functions are the same. As it was shown already in \cite{erdoshajnalmoon} that $sat(n,K_k)=ssat(n,K_k)=n(k-2)-\binom{k-1}{2}$, this determines the saturation number for $K_k$ in the ordered and cyclically ordered case as well.

As saturation functions make sense and so are defined only for non-empty graphs we always assume from now on that the forbidden graph is non-empty.
\subsection{A brief history}

K\'aszonyi and Tuza \cite{kaszonyituza} showed the following dichotomy: for an (unordered) graph $G$, $sat(n,G)=O(1)$ if and only if $G$ has an isolated edge and we have $sat(n,G)=\Theta(n)$ otherwise\footnote{They showed this in the case when we can forbid a set of graphs but assumed that one of them has no isolated vertices. However, for one forbidden graph $G$ the proof can be easily modified also to the case when $G$ has isolated vertices.}. 

Already in the paper of  Erd\H os, Hajnal and Moon \cite{erdoshajnalmoon} the respective question when the forbidden and host graphs are both bipartite was asked. A variant of this for which the extremal problem has a long history is when the two parts of vertices are ordered. As mentioned before, this case is usually phrased using forbidden $0$-$1$ matrices.

Studying the extremal problem for forbidden submatrices in $0$-$1$ matrices has a long history (see \cite{tardos_2019} for an introduction), however, the saturation problem was first investigated only recently by Brualdi and Cao \cite{brualdicao}. Their initial results were quickly followed by the more systematic study of Fulek and Keszegh \cite{01sat} who proved dichotomy in this setting too:  $sat_{\01}(n,G)$ is either $O(1)$ or $\Theta(n)$. Characterizing which graphs (equivalently, $0$-$1$ matrices) belong to which class turned out to be much harder. While many families of matrices belonging to $\Theta(n)$ were found, for $O(1)$ they could only find one example, which was a bipartite matching on $5$ edges (equivalently, a $5$ by $5$ permutation $0$-$1$ matrix). Soon afterwards Geneson \cite{geneson} and Berendsohn \cite{berendsohn1} found infinite many permutation matrices belonging to the class $O(1)$ and then very recently Berendsohn \cite{berendsohn2} gave a complete characterization of permutation matrices belonging to the class $O(1)$: in the bipartite ordered graph setting it can be phrased the following way: a bipartite matching graph $G$ has linear saturation function if and only if it is decomposable into two subgraphs such that each part of the vertices of $G$ is split into two intervals by these subgraphs\footnote{The part of the statement that decomposable graphs have linear saturation function is true even for non-matching graphs and was shown already in \cite{01sat}.}, see Figure \ref{fig:sat01example} for an example of how such a decomposition may look like. Despite this substantial progress, a complete characterization for non-permutation matrices is still widely unknown. In contrast to this, already in \cite{01sat} it was shown with a relatively simple proof that the semisaturation function of $0$-$1$ matrices shows the same dichotomy, and in this case the characterization is also given.

\begin{figure}[h]
	\centering
	\includegraphics[width=0.7\textwidth]{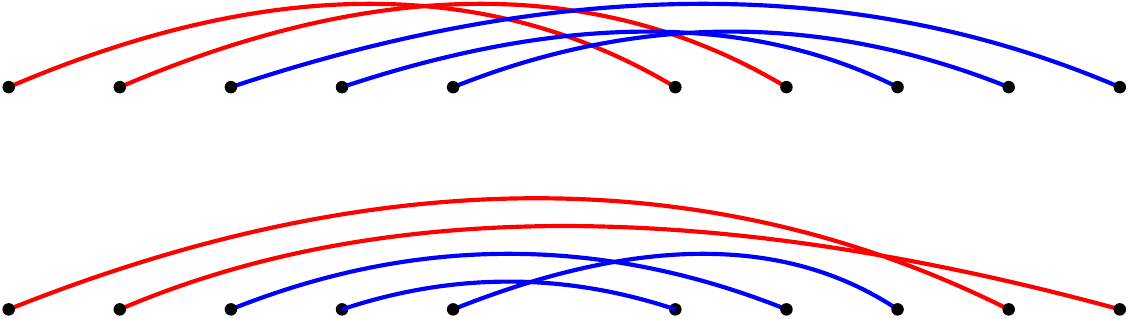}
	\caption{\label{fig:sat01example}The two types of valid decompositions of bipartite matchings into two parts.}
\end{figure}

The study of ordered graphs and their extremal problems is partially motivated by their connection to problems in combinatorial geometry, see, e.g., \cite{Brass2003} which deals with the cyclic case and its relation to the number of unit distances among $n$ points in convex position. The connection is based on the fact that in the cyclic case the vertices can be naturally represented by points in convex position and the edges by straight segments. For the extremal function there is a strong connection between the case of ordered graphs and $0$-$1$ matrices. Namely, the order of magnitude of the extremal function for the ordered case and for the $0$-$1$ matrix case can be apart from each other only by at most an $O(\log n)$ multiplicative factor \cite{Pach2006}.

\subsection{Our results}

Despite all the work on the extremal function of the different variants of ordered graphs, the saturation function was only regarded for the $0$-$1$ matrix case so far. Following up the proposal of P\'alv\"olgyi \cite{dompc}, in this paper we initiate the study of the other two variants, the ordered and cyclically ordered cases, which from a graph theoretic point of view may be even more natural than the $0$-$1$ matrix case.

Motivated by previous results on dichotomy we prove that dichotomy holds also in these two cases and we also give infinite many examples for both classes in both cases. The proofs of this in the three ordered settings are similar but also have significant differences. In particular, for the saturation function one cannot argue that the ordered and ordered bipartite cases are only $O(\log n)$ apart (unlike for the extremal function). Indeed, it turns out that if $G$ has interval chromatic number at most two then $sat_<(n,G)=\Theta(n)$, so bipartite ordered graphs all have linear saturation function in the ordered case, in contrast to their more interesting behaviour in the ordered bipartite case.

In saturation problems it is a primary problem to characterize which objects have bounded saturation functions, in some cases it is non-trivial even to find one which has bounded or non-bounded saturation function. For example, in the case of forbidding posets in the Boolean poset it was shown that the chains have bounded saturation functions \cite{Gerbner2013}. In case we forbid induced posets then it turns out that every poset has a bounded saturation function \cite{keszegh2020induced}. In case of $0$-$1$ matrices, as mentioned already, it was non-trivial to find even one matrix with bounded saturation function \cite{01sat}. This motivated P\'alv\"olgyi \cite{dompc} to ask the problem to find at least one ordered graph with bounded saturation function. We answer this question by finding an infinite class of ordered (resp. cyclically ordered) graphs with bounded saturation function.

We have seen that in the $0$-$1$ matrix case non-decomposable permutation matrices are the prime examples having bounded saturation function \cite{berendsohn2}. When regarded as graphs, permutation matrices correspond to matchings. This motivates us to concentrate on matchings in the ordered case, specifically the ones with interval chromatic number at least three, as we have mentioned that any ordered graph with interval chromatic number at most two has linear saturation function.\footnote{Matchings is usually a natural class to consider in such situations, see also, e.g., the case of ordered graph ramsey theory where they usually concentrate on ordered matchings \cite{conlon2016ordered}.} We define a natural subclass of them which we call \emph{linked matchings} and we start their systematic study. For linked matchings with at most three links and at most one minedge inside each link edge (for exact definitions see later) we determine their saturation function. Most of these small cases have very specific proofs\footnote{In some cases it was surprisingly hard to find constructions to show that the saturation function is bounded. It is worth to compare this with the ordered bipartite case, where also for a while there was only one known such graph.}. Generalizing one of these cases we get an infinite family of ordered graphs with bounded saturation function. However, in general, we are not yet able to characterize the linked matchings that have linear saturation function. 

In both the ordered and cyclically ordered case we also consider the semisaturation functions. Similar to the $0$-$1$ matrix case, dichotomy again holds and even the characterization is relatively simple.

The paper is structured as follows. Section \ref{sec:ordered} contains our results about saturation of ordered graphs. First we prove dichotomy then give several examples of graph classes with linear saturation function. Then we define the notion of a witness, which we use to show in Section \ref{sec:linked} an infinite family of linked matchings with bounded saturation function, including a treatment of linked matchings with at most three links. Section \ref{sec:cyclic} contains our results about saturation of cyclically ordered graphs. Most importantly, we prove dichotomy then give infinite many examples for both possible behaviours of the saturation function. In Section \ref{sec:ssat} we solve the semisaturation problem. In Section \ref{sec:discussion} we compare our knowledge of the three settings (bipartite ordered, ordered, cyclically ordered) and conclude the paper with a list of open problems.

\section{Saturation of ordered graphs}\label{sec:ordered}

Before phrasing and proving our results, let us fix some conventions we will use throughout the paper. Without loss of generality, we identify the vertices of an ordered graph on $n$ vertices with the positive integers $1,2, \dots ,n$ so that the order of the vertices is the same as the order of the corresponding positive integers. We imagine the vertices lying on a horizontal line placed from left to right according to their underlying order.

We can refer to edges by their two endvertices, e.g., $e=uv$ is an edge, 
where we always assume $u<v$. We write $l(e)=u$ and $r(e)=v$ for the left and right endvertices, respectively, of the edge $e=uv$. Moreover, for a vertex $v$ the vertices $v-1$ and $v+1$ are called the left and right neighbor of $v$, respectively. We say that a vertex $w\ne u,v$ is between the vertices $u$ and $v$ if $u<w<v$ and outside $uv$ otherwise. We say that an edge $uv$ is \emph{covered} by another edge $u'v'$ if $u'\le u<v\le v'$. We say that an edge $uv$ is \emph{strictly covered} by another edge $u'v'$ if $u'< u<v< v'$.

The following two types of edges of $G$ turn out to have an important role in determining if $G$ has bounded saturation function:

\begin{defi}
	An edge $uv$ of $G$ is a \emph{minedge} if there is no vertex between $u$ and $v$ and both $u$ and $v$ have degree one, i.e., it is an isolated edge connecting neighboring vertices. 
	
	An edge $uv$ of $G$ is a \emph{superedge} if there exists an edge $xy$ such that $u<x<y<v$, i.e., $uv$ strictly covers $xy$.
\end{defi}

\begin{defi}
    A graph $G$ is \emph{separable} if it can be split into non-empty edge-disjoint graphs $G_1$ and $G_2$ such that for all edges $u_1v_1$ of $G_1$ and $u_2v_2$ of $G_2$ we have $u_1<v_1<u_2<v_2$.
    
    A graph $G$ is \emph{nested} if it can be split into non-empty edge-disjoint graphs $G_1$ and $G_2$ such that for all edges $u_1v_1$ of $G_1$ and $u_2v_2$ of $G_2$ we have $u_1<u_2<v_2<v_1$, i.e., the edges of $G_2$ are strictly covered by the edges of $G_1$. See Figure \ref{fig:sepnest} for examples.
\end{defi}

\begin{figure}[h]
	\centering
	\includegraphics[width=0.75\textwidth]{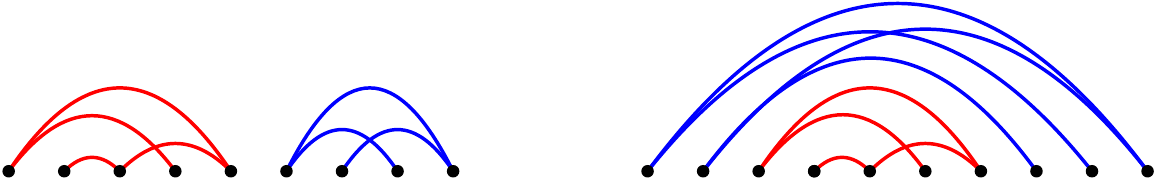}
	\caption{\label{fig:sepnest}A separable and a nested graph.}
\end{figure}

Note that in a nested graph $G$ the outer graph (i.e., $G_1$) must have interval chromatic number two.

We first show the dichotomy of the function $sat_{<}$:

\begin{thm}\label{thm:ordereddich}
	Given an ordered graph $G$, we either have $sat_{<}(n,G)=O(1)$ or $sat_{<}(n,G)=\Theta(n)$.
\end{thm}
\begin{proof}
	Let $H_n$ be the host graph saturating $G$ on $n$ vertices with $sat_{<}(n,G)$ edges. If $G$ has no isolated vertices then if there exists an $n_0$ such that $H_{n_0}$ contains two adjacent isolated vertices, then we can multiply these vertices to get host graphs of size $n>n_0$ with the same number of edges, showing that $sat_{<}(n,G)\le sat_{<}(n_0,G)=O(1)$ for $n\ge n_0$. If $G$ contains isolated vertices then instead of two we require $|V(G)|$ many consecutive isolated vertices and get to the same conclusion.
	
	Thus either $sat_{<}(n,G)=O(1)$ or there are no two (resp. $|V(G)|$ many) consecutive isolated vertices in the host graphs, which implies $sat_{<}(n,G)=\Omega(n)$.
	
	We are left to prove that $sat_{<}(n,G)=O(n)$ always holds. For that take an arbitrary edge $uv$, $u<v$, of $G$ s.t. there is no other edge $u'v'$ such that $u\le u'<v'\le v$.	Denote the number of vertices preceding $u$ by $a$, the number of vertices between $u$ and $v$ by $b$, and the number of vertices succeeding $v$ by $c$ (in particular $a+b+c+2=|V(G)|$).

	For $n>|V(G)|$ let $H_n$ be a graph on $n$ vertices such that its first $a$ and last $c$ vertices are connected with every vertex and we also add all the edges $ij$ with at most $b-1$ vertices between $i$ and $j$. $H_n$ has $O(n)$ edges. We claim that it is a host graph saturating $G$. 
	
	First, $H$ avoids $G$. Assume on the contrary. Take a copy of $G$ contained in $H$ and let $i$ and $j$ be the vertices playing the roles of $u$ and $v$ in this copy of $G$, respectively (the vertex set of $H$ is $[n]$). If $i\le a$ then we cannot find in $H$ the $a$ vertices preceding $u$ in $G$. If $j\ge n-c+1$ then we cannot find in $H$ the $c$ vertices succeeding $v$ in $G$. Finally, if $j-i<b$ then we cannot find in $H$ the $b$ vertices between $u$ and $v$ in $G$. All cases lead to a contradiction.
	
	Second, adding an arbitrary edge $ij$ to $H$ we have that $i>a$, $j<n-c+1$ and $j-i\le b$ and so we can take $i,j$, the first $a$ vertices, the last $c$ vertices and $b$ vertices between $i$ and $j$ in $H$ to find a copy of $G$ in $H+uv$.
\end{proof}	

Although we are mostly concerned about the order of magnitude of the saturation function, it is worth to note that actually when $sat_{<}(n,G)=O(1)$ then there exists a number $n_0$ such that $sat_{<}(n,G)=sat_{<}(n_0,G)$ for $n\ge n_0$. This follows from the fact that for big enough $n$ we necessarily have many consecutive isolated vertices in a host with $sat_{<}(n,G)$ edges and adding or removing an isolated vertex from these gives us another saturated host.\footnote{This reasoning was used already for graphs in \cite{kaszonyituza}.}

\begin{obs}\label{obs:single}
	If $G$ is a single edge then $sat_{<}(n,G)= 0$.
\end{obs}

\begin{proof}
	The graph on $n$ vertices and no edges is a suitable host graph saturating $G$.
\end{proof}

\begin{claim}\label{claim:minedge}
	If $G$ contains no minedge then $sat_{<}(n,G)= \Theta(n)$.
\end{claim}

\begin{proof}
	In a host graph $H$ on $n$ vertices saturating $G$ we cannot have two adjacent isolated vertices, as otherwise we could connect these edges without introducing a copy of $G$ (as this edge must play the role of a minedge, but $G$ has no minedge). This implies that  $sat_{<}(n,G)= \Omega(n)$ which together with Theorem \ref{thm:ordereddich} finishes the proof.
\end{proof}

The next theorem is somewhat similar to the result from \cite{01sat} mentioned in the introduction saying that decomposable $0$-$1$ matrices have a linear saturation function, compare Figure \ref{fig:sat01example} and Figure \ref{fig:sepnest}.

\begin{thm}\label{thm:sep}
	If the ordered graph $G$ is separable or nested then $sat_{<}(n,G)=\Theta(n)$.
\end{thm}
\begin{proof}
	We start with the case when $G$ is separable.
	Let $G_1$ and $G_2$ be the graphs that show that $G$ is separable.
	Let $H$ be a host graph on $n$ vertices saturating $G$ that has $sat_{<}(n,G)$ edges. If $H$ contains no isolated vertex then it has $\Omega (n)$ edges, as required. Otherwise, take an isolated vertex $x$ of $H$ and connect it to the first vertex of $H$. This must introduce a copy of $G$ in which this edge must play the role of some edge of $G_1$. This implies that there is a copy of $G_2$ in $H$ to the right from $x$. Similarly, connecting $x$ to the last vertex shows that there is a copy of $G_1$ in $H$ to the left from $x$. These two together form a copy of $G$, a contradiction. Thus, $H$ has $\Omega (n)$ edges which together with Theorem \ref{thm:ordereddich} finishes the proof.

	We continue with the case when $G$ is nested.
	Let $G_1$ and $G_2$ be the graphs that show that $G$ is nested. Notice that $G_1$ is a bipartite graph with parts $A$ and $B$ such that the vertices in $A$ precede the vertices in $B$ (i.e., $G$ has interval chromatic number two).
	Let $H$ be a host graph on $n$ vertices saturating $G$ that has $sat_{<}(n,G)$ edges. Adding an edge to $H$ introduces a copy of $G$. If $H$ has no two adjacent isolated vertices then $sat_{<}(n,G)=\Omega(n)$ which together with Theorem \ref{thm:ordereddich} finishes the proof. Otherwise, let $w'<w$ be two consecutive isolated vertices in $H$. Adding an edge to $H$ introduces a copy of $G$. If we add an edge to $H$ between $w'$ and $w$, it must play the role of an edge from $G_2$. If we add an edge to $H$ between the first vertex and $w$, it must play the role of an edge from $G_1$. This implies (by going from the leftmost vertex one by one towards $w$) that  there are two adjacent vertices $u,v$ with $u<v<w$, such that adding $uw$ we get a copy of $G$ in which $uw$ plays the role of an edge from $G_1$, while adding $vw$ we get a copy of $G$ in which $vw$ plays the role of an edge from $G_2$. This implies that in $H$ there is a copy of $G_2$ between $u$ and $w$, and a copy of $G_1$ such that the vertices corresponding to $A$ precede $v$ and the vertices corresponding to $B$ succeed $w$ in $H$. These together form a copy of $G$, a contradiction. 
\end{proof}	

\begin{cor}\label{cor:coverminedge}
	If $G$ has a minedge (strictly) covered by every other edge of $G$, then $sat_{<}(n,G)=\Theta(n)$.
\end{cor}

Note that covering and strictly covering a minedge is equivalent as there are no other edges incident to the endvertices of the minedge by definition.

\begin{cor}\label{cor:intchrom}
	If the interval-chromatic number of $G$ is two, then $sat_{<}(n,G)=\Theta(n)$.
\end{cor}

\begin{proof}
	If $G$ does not contain a minedge then we are done by Claim \ref{claim:minedge}. Otherwise $G$ contains a minedge. Then this minedge must be strictly covered by every other edge and so by Corollary \ref{cor:coverminedge} we are done.
\end{proof}

As mentioned in the introduction, Corollary \ref{cor:intchrom} implies that the graphs that correspond to permutation matrices all have linear saturation function in the ordered case.

\begin{claim}\label{claim:firstneighbors}
	If every neighbor of the first vertex (or the last vertex) in $G$ has degree greater than one then $sat_{<}(n,G)=\Theta(n)$. \footnote{We note that Theorem \ref{thm:semiord} implies that such a $G$ has $ssat_{<}(n,G)=\Theta(n)$, from which Claim \ref{claim:firstneighbors} also follows.}
\end{claim}

\begin{proof}
	Let $H$ be a host graph saturating $G$ on $n$ vertices that has $sat_{<}(n,G)$ edges. If $H$ has no isolated vertices then it has $\Theta(n)$ vertices, as required. Otherwise, take the first vertex, $v$, of $H$. Connecting $v$ to an isolated vertex $u$ creates a copy of $G$ which uses this new edge. Thus, $v$ must be the first vertex in this copy of $G$ and $u$ must be the neighbor of the first vertex. However, the degree of $u$ is one in $H$ and so at most one in this copy of $G$, contradicting the assumption of the lemma.
\end{proof}

\begin{lem}\label{lem:minsupdeg}
    Let $M\subseteq E(G)$ be the set of all minedges of $G$, let $S\subseteq E(G)$ be the set of all superedges of $G$ and let $L=\{uv \in E(G): deg(u)= 1\}$ \footnote{Notice that we can replace $L$ by $R=\{uv \in E(G): deg(v)= 1\}$ and obtain the same result for $R$.}. If $L\subseteq S \cup M$ then $sat_{<}(n,G)=\Theta(n)$.
\end{lem}

\begin{proof}
    If $G$ is separable then we are done by Claim \ref{thm:sep}. So we can assume that $G$ is not separable which implies that every minedge of $G$ is strictly covered by some edge (which is thus a superedge).
    
    Let $H$ be a host graph saturating $G$ on $n$ vertices that has $sat_{<}(n,G)$ edges. If there are no two adjacent isolated vertices in $H$ then it has $\Theta(n)$ edges, we are done. Otherwise, take two adjacent isolated vertices $u<v$ of $H$. Connecting them creates a copy of $G$, clearly $uv$ has to be a minedge in this copy of $G$. Connecting $u$ to the last vertex also creates a copy of $G$, in which copy this new edge cannot be a minedge (as it cannot be strictly covered by some other edge). Connecting $u$ one-by-one to the vertices that are bigger than $u$, there is a first (i.e., leftmost) vertex $w$ after $u$ such that adding $uw$ to $H$ creates a copy of $G$ in which $uw$ is not a minedge. Since the degree of $u$ must be $1$ in this copy of $G$, $uw$ corresponds to an edge of $G$ which is in $L$. As it is not a minedge of $G$, it has to be a superedge by the assumption of the claim. This implies that there is an edge $xy$ strictly covered by it, i.e., $u<x<y<w$. Let $w'$ be the vertex preceding $w$. Take the copy of $G$ that is created when we add the edge $uw'$ to $H$. In this copy $uw'$ must play the role of a minedge, thus by replacing $uw'$ with $xy$ we get a copy of $G$ that is already in $H$, a contradiction.
\end{proof}

\begin{thm}\label{thm:minsup}
    If every edge of $G$ is a minedge or a superedge then $sat_{<}(n,G)=\Theta(n)$.
\end{thm}

\begin{proof}
    As $L\subseteq E(G)= S \cup M$, from Lemma \ref{lem:minsupdeg} we obtain that $sat(n,G)=\Theta(n)$, as required.
\end{proof}

So far our results gave families of ordered graphs that have linear saturation function. In the next section we will give infinite many graphs that have bounded saturation function. The main general tool for proving that an ordered graph has a bounded saturation function is the notion of a witness. Similar to witnesses and explicit witnesses of $0$-$1$ matrices defined in \cite{berendsohn1}, we define the respective notions for ordered graphs:

\begin{defi}
	Given an ordered graph $G$ with no isolated vertices, an ordered graph $H$ is an \emph{explicit witness} of $G$ if $H$ saturates $G$ and has two consecutive isolated vertices. $H$ is a \emph{witness} of $G$ if $H$ has two consecutive isolated vertices and adding an arbitrary new edge incident to any of these two isolated vertices creates a copy of $G$.
	
	If $G$ is an ordered matching then we require only one isolated vertex in the above definitions.	
\end{defi}

Witnesses, as their name suggests, witness that a graph has a bounded saturation function. First, it is trivial by definition that if the saturation function of $G$ is $O(1)$ then for $n$ big enough a host graph of size $n$ saturating $G$ with $sat_<(n,G)$ edges must be an explicit witness for $G$.\footnote{We implicitly used this fact already in the proof of Theorem \ref{thm:ordereddich}.} Thus bounded saturation function for $G$ implies that $G$ has an (explicit) witness. Now we prove that the opposite is also true.\footnote{For simplicity we did not define witnesses for graphs with isolated vertices. Nevertheless, the definition and Lemma \ref{lem:witness} could be easily extended to handle this case as well by requiring more consecutive isolated vertices.}

\begin{lem}\label{lem:witness}
 	Let $G$ be an ordered graph without isolated vertices such that there exists an (explicit) witness of $G$. Then $sat_{<}(n,G)=O(1)$.
\end{lem}

\begin{proof}
  First we prove that if an ordered graph $G$ has a witness $H'$ with two consecutive isolated vertices then it also has an explicit witness $H$. Indeed, we can greedily add edges to $H'$ until we get a graph $H$ that saturates $G$. Due to the definition of a witness, in $H$ the two isolated vertices are still isolated and thus $H$ indeed is an explicit witness. The case when $G$ is an ordered matching is analogous, we just need to replace two isolated vertices with one in the argument.
  
  Thus from now on we can assume that $G$ has an explicit witness $H$. Take an arbitrary number $n\ge |V(G)|$. If we replace in $H$ two adjacent isolated vertices by an interval $I$ of $n-|V(G)|+2$ consecutive isolated vertices, the graph $H_n$ we get has $n$ vertices and still avoids $G$. Further, the fact that by adding an arbitrary edge to $H$ we get a copy of $G$ implies that if we add any edge to $H_n$ we obtain a copy of $G$. In $H_n$ the vertices in $I$ are isolated and thus $H_n$ has $O(1)$ edges. We can conclude that $H_n$ is saturating $G$ and has $O(1)$ edges, which implies that $sat_{<}(n,G)=O(1)$, as claimed. 
   
  Note that we needed two adjacent isolated vertices to guarantee that when an edge is added inside $I$ then a copy of $G$ is created. However, it is easy to see that if $G$ is a matching then this is guaranteed already if $H$ has one isolated vertex which we replace with an interval $I$ of isolated vertices, thus the part of the lemma about matchings follows as well.
\end{proof}

\subsection{Saturation of linked matchings}\label{sec:linked}

This section is dedicated to the study of the saturation function of linked matchings. After defining them, first we prove for an infinite family of linked matchings that they have linear saturation function, then we concentrate on small cases and determine which linked matchings with at most three links have bounded saturation function. Finally we generalize some of our constructions to show for an infinite family of linked matchings that they have bounded saturation function.

\begin{defi}
Let $L_k$ be the ordered matching on vertex set $[2k]$ and edge set $\{(2i-1)(2i):i=1,\dots,k\}$. 

Let $\Gamma_k$ be the ordered matching on vertex set $[2k]$ and edge set $\{(1)(3)\}\cup\{(2i)(2i+3):i=1,\dots,k-2\}\cup\{(2k-2)(2k)\}$. See Figure \ref{fig:gamma} for illustrations.
\end{defi}

\begin{defi}
$\Gamma_{\{m_1,m_2,...,m_k\}}$ denotes the matching that we obtain from $\Gamma_k$ if for every $i$ we add $2m_i$ new vertices between the endvertices of the $i$th edge of $\Gamma_k$ (when ordered by their left endvertex) and outside the rest of the edges of $\Gamma_k$ and put a copy of $L_{m_i}$ on these $2m_i$ vertices. $\Gamma_{\{m_1,m_2,...,m_k\}}$ is called a \emph{linked matching}, while the edges of the underlying $\Gamma_k$ are called its \emph{link edges}.\footnote{We regard these link edges as being ordered according to the order of their left endvertices. Notice that the non-link edges are all minedges.} See Figure \ref{fig:gamma} for illustrations.
\end{defi}

Note that $L_k$ is separable and thus has linear saturation function by Claim \ref{thm:sep} for $k\ge 2$.

\begin{figure}[h]
	\centering
	\includegraphics[width=0.7\textwidth]{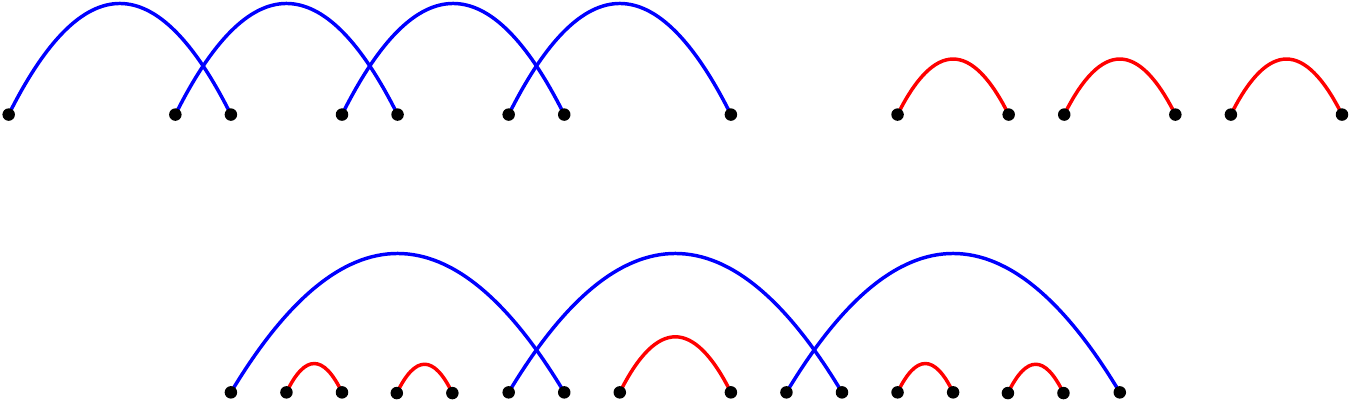}
	\caption{\label{fig:gamma} Examples: $\Gamma_4$, $L_3$ and 
	$\Gamma_{\{2,1,2\}}$.}
\end{figure}

\smallskip
The next claim follows immediately from Theorem \ref{thm:minsup}:

\begin{claim}\label{claim:allnonzero}
	$sat_{<}(n,\Gamma_{\{m_1,m_2,\dots,m_k\}})=\Theta(n)$ for $k>0$ and $m_i>0$ ($1\le i\le k$).
\end{claim}

\begin{claim}\label{claim:firstnonzero}
     $sat_{<}(n,\Gamma_{\{m,0,0,...,0\}})=sat_{<}(n,\Gamma_{\{0,,...,0,0,m\}})=\Theta(n)$ for $m \in \mathbb{N}$.
\end{claim}

\begin{proof}
   Let $G=\Gamma_{\{m,0,0,...,0\}}$, a linked matching with $k>1$ link edges and $m$ minedges. If $m=0$ then $G$ does not contain a minedge and thus we are done by Claim \ref{claim:minedge}. Thus from now on we assume that $m>0$. Let $H$ be a host graph on $n$ vertices saturating $G$ with $sat_<(n,G)$ edges. If $H$ has no two consecutive isolated vertices, then $sat_<(n,G)=\Omega(n)$, as required. Otherwise, there are consecutive isolated vertices $w',w$ in $H$. Adding the edge $w'w$ creates a copy of $H$ in which this edge must be a minedge. Adding the edge connecting $w$ with the first vertex creates a copy of $H$ in which this edge cannot be a minedge (as in $G$ no minedge is incident to the first vertex of $G$). Thus, if we connect one-by-one $w$ to vertices to the left from $w$ there will be a leftmost vertex $u$ such that $uw$ is a minedge in a copy $G'$ of $G$ created by adding $uw$ to $H$. 

   Since all minedges of $G$ are contained inside the first link edge, the first link of $G'$ is an edge $xy$ of $H$ such that $x<u$. Now adding the edge $xw$ to $H$ creates a copy $G''$ of $G$. Notice that $xw$ cannot be a minedge in $G''$ due to our choice of $u$. Thus $xw$ is the $i$th link edge of $G''$ for some $i$. Now let $G'''$ be the graph formed by the first $i-1$ link edges of $G''$, the $m$ minedges of $G''$ and the first $k-i+1$ link edges of $G'$. This $G'''$ is a copy of $G$ in $H$, a contradiction. See Figure \ref{fig:firstnonzero} for an illustration.
\end{proof}

\begin{figure}[h]
	\centering
	\includegraphics[width=0.75\textwidth]{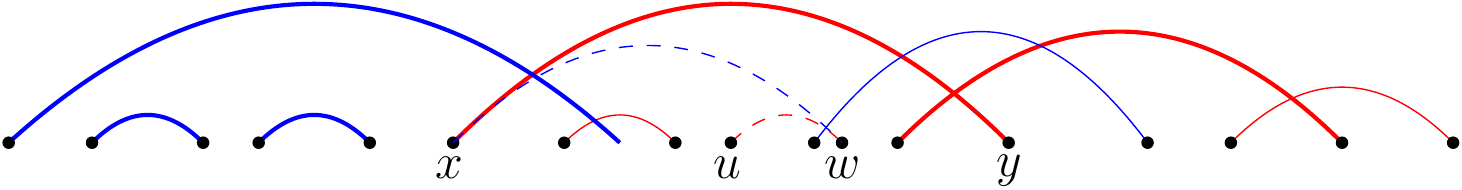}
	\caption{\label{fig:firstnonzero} Proof of Claim \ref{claim:firstnonzero} 
for $G=\Gamma_{\{2,0,0\}}$. Edges not in $H$ are drawn dashed. $G'$ (red) and 
	$G''$ (blue) in $H$ together contain a copy of $G$ (bold) in $H$.}
\end{figure}

\begin{cor}
	$sat_{<}(n,\Gamma_k)=\Theta(n)$ for $k\ge 2$. 	
\end{cor}

Next we concentrate on linked matchings with $k\le 3$ links. First, $sat_<(n,\Gamma_{\{m\}})=0$ if $m=0$ by Observation \ref{obs:single} and $\Theta(n)$ otherwise by Claim \ref{claim:allnonzero}. Second, $sat_<(n,\Gamma_{\{m_1,m_2\}})=\Theta(n)$. Indeed, this follows from Claim \ref{claim:firstnonzero} if at least one of $m_1,m_2$ equals to $0$ and from Claim \ref{claim:allnonzero} otherwise. The problem becomes much more interesting when $k=3$.

The following is again a direct corollary of Claim \ref{claim:allnonzero} and Claim \ref{claim:firstnonzero}:

\begin{cor}
	 $sat_{<}(n,\Gamma_{\{m,0,0\}})=sat_{<}(n,\Gamma_{\{0,0,m\}})=\Theta(n)$ for $m>0$. 	
	 
	  $sat_{<}(n,\Gamma_{\{m_1,m_2,m_3\}})=\Theta(n)$ for $m_1,m_2,m_3>0$. 
\end{cor}

For the remaining cases we only consider the subcases when each $m_i$ is either 
$0$ or $1$ and prove that all of these have bounded saturation function:

\begin{thm}\label{thm:gamma3}
	The saturation functions of $\Gamma_{\{0,1,0\}}, \Gamma_{\{1,0,1\}},\Gamma_{\{1,1,0\}}$ and $\Gamma_{\{0,1,1\}}$ are bounded.
\end{thm}

For each of these graphs the hard part was to find a witness $H$ for a graph $G$. 
To prove that they are indeed witnesses involves only some case analysis which 
could be easily done even by a computer program, showing first that they 
avoid $G$, second that adding any edge to the isolated vertex of $H$ creates a 
copy of $G$. Instead of using a computer, we do the case analysis explicitly and as efficiently as possible. The following three claims together imply Theorem \ref{thm:gamma3}.

\begin{figure}[h]
\centering
\includegraphics[width=0.6\textwidth]{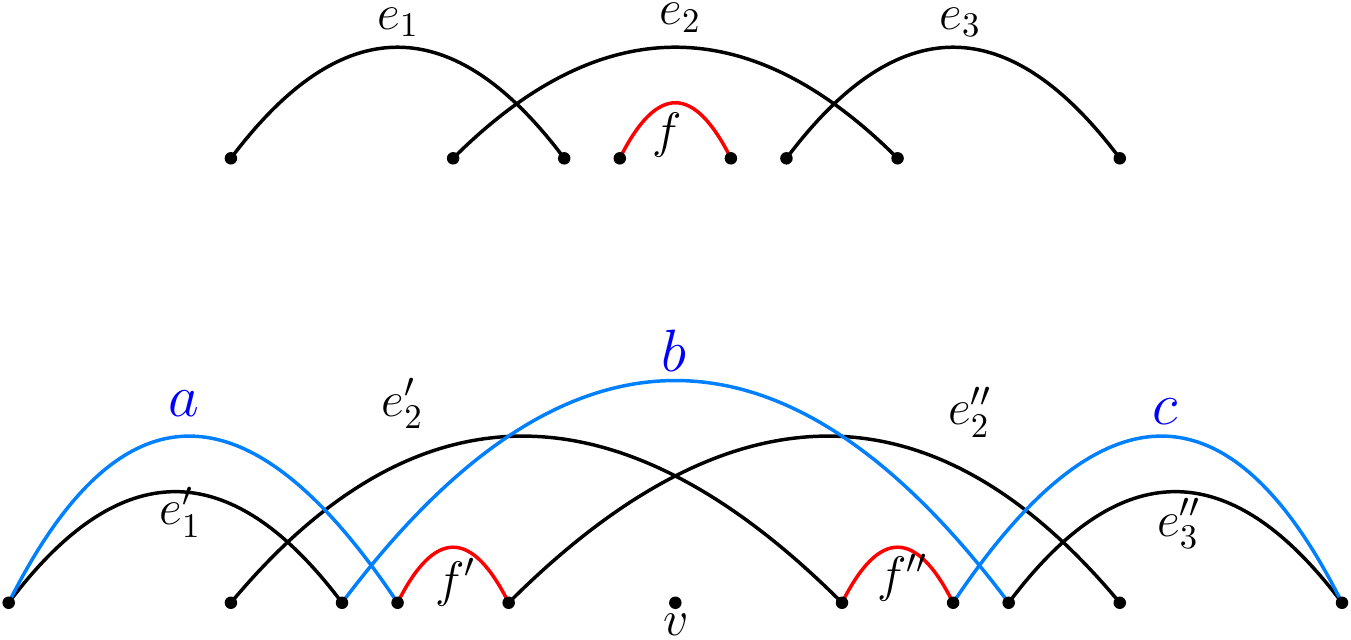}
\caption{\label{fig:010} $\Gamma_{\{0,1,0\}}$ and its witness graph $H$.}
\end{figure}

\begin{claim}
     $sat_{<}(n,\Gamma_{\{0,1,0\}})=O(1)$. 
\end{claim}

\begin{proof}
    Let $G=\Gamma_{\{0,1,0\}}$, its link edges are denoted by $e_1,e_2,e_3$, in 
    this order and its minedge is denoted by $f$. Furthermore let $H$ be the 
    graph drawn on Figure \ref{fig:010} bottom with edges named according to 
    the figure. We claim that $H$ is a witness for $G$, which by Lemma 
    \ref{lem:witness} implies the statement of the claim. 
    
    First we show that $H$ avoids $G$. Assume on the contrary that there is a copy of $G$ in $H$, then there is an edge $e=u_1u_2 \in E(H)$ such that it has the same role as edge $e_2$ in $G$, in particular it is a superedge. Notice that $|\{v \in V(G): v> l(e_2) \}|=|\{v \in V(G): v < r(e_2)\}|=6$ which implies that $u_1$ (resp. $u_2$) has to have at least $6$ vertices on its right (resp. left) side in $H$. The only edge in $H$ that satisfies this property is the edge $b$. Then the minedge $f$ can be mapped to either $f'$ or $f''$. If it is mapped to $f'$ (resp. $f''$) then $e_1$ (resp. $e_3$) cannot be mapped to an edge of $H$ since there is no vertex between $l(b)$ and $l(f')$ (resp. $r(b)$ and $r(f'')$).
    
    It remains to show that adding an arbitrary edge incident to $v$ in $H$ 
    creates a copy of $G$. Let $H'$ be the subraph of $H$ with edge set 
    $\{e'_1,e'_2,f'\}$ and $H''$ be the subgraph of $H$ with edge set 
    $\{e''_2,e''_3,f''\}$. If we connect $v$ to some $w$ with $w>l(f'')$ then 
    $H'+vw$ contains a copy of $G$. Similarly, if $w<r(f')$ then $H''+wv$ 
    contains a copy of $G$. Finally if $w=l(f'')$ or $w=r(f')$ then we can 
    define $H^{*}$ to be the subgraph with edge set $\{a,b,c\}$ and see that 
    $H^{*}+vw$ contains a copy of $G$, concluding the proof.
\end{proof}

\begin{claim}
    $sat_{<}(n,\Gamma_{\{1,0,1\}})=O(1)$. 
\end{claim}

\begin{proof}
    Let $G=\Gamma_{\{1,0,1\}}$, its link edges are denoted by $e_1,e_2,e_3$, in 
    this order and its minedge covered by $e_1$ (resp. $e_3$) is denoted by 
    $f_1$ (resp. $f_2$). Furthermore let $H$ be the graph drawn on Figure 
    \ref{fig:101} bottom with vertices named according to the figure. We claim 
    that $H$ is a witness for $G$, which by Lemma \ref{lem:witness} implies the 
    statement of the claim. 
	
	First we show that adding an arbitrary edge incident to $v$ in $H$ creates 
	a copy of $G$. As both $G$ and $H$ are symmetric, it is enough to check 
	that connecting $v$ to an arbitrary vertex $w$ to the right from $v$ 
	introduces a copy of $G$. We have three cases:
    
    Case 1: $v < w < v'_3$.
    The following map gives a copy of $G$: $e_1,e_2,e_3$ are mapped to  
    $v_{10}v'_4,v'_3v'_7$ and $v'_6v'_{10}$, respectively, while 
    $f_1$ and $f_2$ are mapped to $vw$ and $v'_8v'_9$, respectively. 
    
    Case 2: $w = v'_3$. 
    The following map gives a copy of $G$: $e_1,e_2,e_3$ are mapped to 
    $v_7v'_1,vw$ and $v'_2v'_6$, respectively, while $f_1$ and $f_2$ 
    are mapped to $v_8v_9$ and $v'_4v'_5$, respectively. 
    
    Case 3: $w > v'_3$. 
    The following map gives a copy of $G$: $e_1,e_2,e_3$ are mapped to 
    $v_5v_{10}, v_9v'_1$ and $vw$, 
    respectively, while $f_1,f_2$ are mapped to $v_6v_7$ and $v'_2v'_3$, 
    respectively. 
    
    \begin{figure}[h]
    \centering
    \includegraphics[width=0.7\textwidth]{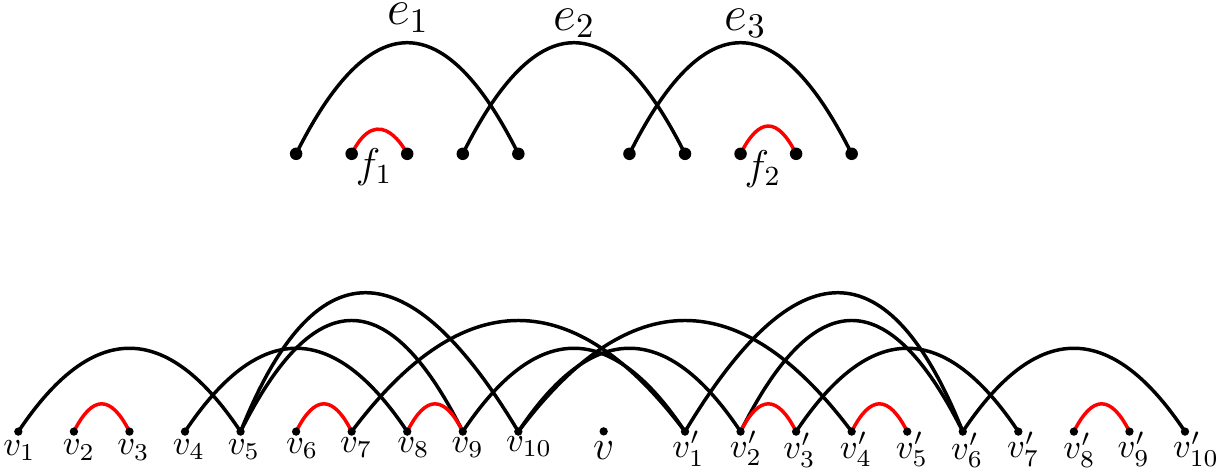}
    \caption{\label{fig:101} Graph $\Gamma_{\{1,0,1\}}$ and its witness graph $H$.}
    \end{figure}
  
    We are left to show that $H$ avoids $G$. Assume on the contrary that there 
    is a copy of $G$ in $H$. We first check if the edges $uw$ such that 
    $u<v<w$ are contained in a copy of $G$. These edges are $v_7v'_1, 
    v_9v'_1, v_{10}v'_2, v_{10}v'_4$. By symmetry it is enough to check only 
    $v_7v'_1$ and $v_9v'_1$.
    
    Assume first that $v_9v'_1$ is in a copy of $G$. Notice that there is only 
    one non-isolated vertex between $v_9$ and 
    $v'_1$. So, this edge must be a minedge in this copy as for all 
    non-minedges of $G$ there are at least two non-isolated vertices between 
    its endvertices. On the other hand, there is no edge in $H$ that strictly 
    covers $v_9v'_1$, so it cannot be a minedge in this copy, contradiction. 
    
    Assume second that $v_7v'_1$ is in a copy of $G$. As $v_7v'_1$ is not 
    strictly covered by any edge in $H$, it cannot be a minedge in this copy of 
    $G$, thus it is a link edge of $G$. Notice that 
    $v_8v_9$ is the only 
    candidate for the minedge contained in this link edge. Since there is no 
    vertex between $v_7$ and $v_8$, it cannot be $e_3$ in a copy of $G$. Assume 
    now that $v_8v_9$ is $e_1$ in this 
    copy of $G$. In that case the only possibility for $e_2$ is the edge 
    $v_{10}v'_4$. Then every edge that is a candidate for $e_3$ has as a right 
    vertex either 
    $v_6'$ or $v_7'$. However, since there is no edge between $v_4'$ and 
    $v'_7$, we cannot map $f_2$ to an edge of $H$, a contradiction.
    Assume finally that $v_7v'_1$ is $e_2$ in this copy of $G$. Similar to the 
    previous case, all edges that are candidates for $e_1$ have as a left 
    vertex either $v_4$ or $v_5$. However, since there is no edge between $v_4$ 
    and $v_7$, we cannot map $f_1$ to an edge of $H$, a contradiction. 
    
    So far we have shown that no edge $uw$ such that $u<v<w$ can be an edge of 
    a 
    copy of $G$ in $H$. As $G$ is not a separable graph, it follows that either 
    every vertex of a copy of $G$ is left from $v$ or every vertex is right 
    from $v$. By symmetry we can assume it is left from $v$. There are exactly 
    $10$ vertices left from $v$ in $H$ and also $|V(G)|=10$. 
    it means that there is a bijection $\phi$ between $V(G)$ and 
    $\{v_1,v_2,...,v_{10}\}$. In particular $e_3$ must be mapped to an edge 
    between $v_6$ and $v_{10}$ but there is no such edge in $H$, a 
    contradiction.
    
    Thus there is no copy of $G$ in $H$, which finishes the proof.
\end{proof}

\begin{figure}[h]
\centering
\includegraphics[width=0.7\textwidth]{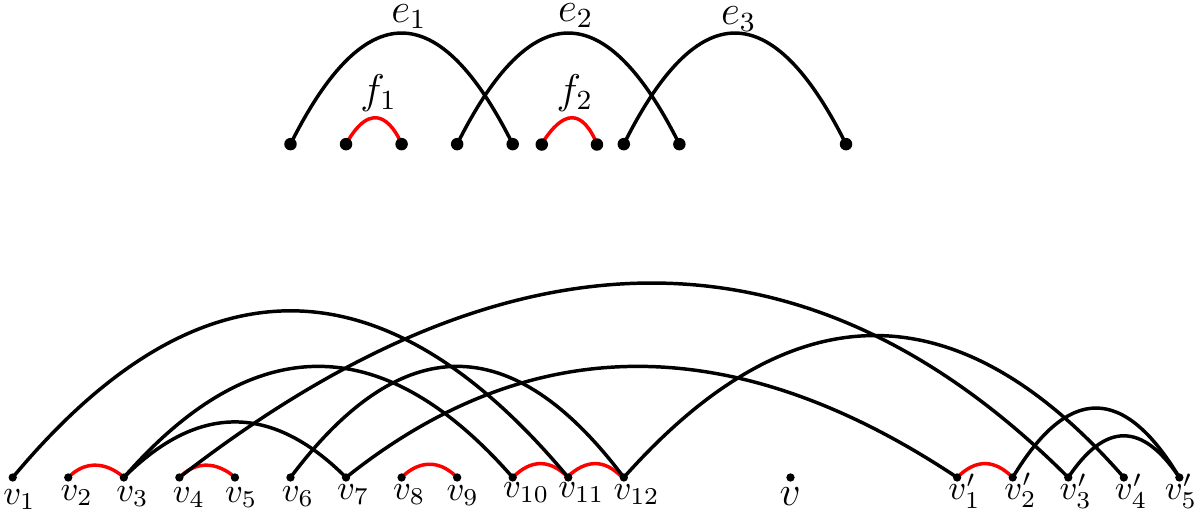}
\caption{\label{fig:110} Graph $\Gamma_{\{1,1,0\}}$ and its witness graph $H$.}
\end{figure}

\begin{claim}
    $sat_{<}(n,\Gamma_{\{1,1,0\}})=sat_{<}(n,\Gamma_{\{0,1,1\}})=O(1)$. 
\end{claim}

\begin{proof}
	Let $G=\Gamma_{\{1,1,0\}}$, its link edges are denoted by $e_1,e_2,e_3$, in 
	this order and its minedge covered by $e_1$ (resp. $e_2$) is denoted by 
	$f_1$ (resp. $f_2$). Furthermore let $H$ be the graph drawn on Figure 
	\ref{fig:110} bottom with vertices named according to the figure. We claim 
	that $H$ is a witness for $G$, which by Lemma \ref{lem:witness} implies the 
	statement of the claim. 
	
	First we show that $H$ avoids $G$. Assume on the contrary that there is a copy $G_1$ of $G$ in $H$. In this copy $e_3$ is mapped to some edge $uu'$ of $H$. Notice that $7$ vertices precede $l(e_3)$ in $G$, while $9$ vertices precede $r(e_3)$. Thus $u>v_7$ and $u'>v_9$ and also we know that there is at least one vertex between $u$ and $u'$. The only edges that satisfy these properties are: $v_{12}v'_4, v'_2v'_5, v'_3v'_5$. 
	
    Case 1: $e_3$ is mapped to $v_{12}v'_4$. Then $e_2$ must be mapped to one of $\{v_7v'_1, v_4v'_3\}$. Assume that $e_2$ is mapped to $v_7v'_1$ and then $f_2$ must be mapped to $v_8v_9$ or $v_{10}v_{11}$, however neither is possible since then $e_1$ cannot be mapped to any of the edges. Now assume $e_2$ is mapped to $v_4v'_3$ which implies that $e_1$ must be mapped to $v_1v_{11}$ but that is again impossible since there is no room for mapping $f_{2}$ as there are no two vertices connected by an edge between $v_{11}$ and $v_{12}$ (in fact there is not even a vertex here). 
    
    Case 2: $e_3$ is mapped to $v'_2v'_5$. Then $e_2$ must be mapped to one of $\{v_4v'_3,v_{12}v'_4\}$. Similar to the previous case $e_2$ cannot be mapped to $v_4v'_3$, so $e_2$ must be mapped to $v_{12}v'_4$ but then again there is no room for mapping $f_2$ as there are no two vertices connected by an edge between $v_{11}$ and $v_{2'}$.
        
    Case 3: $e_3$ is mapped to $v'_2v'_5$. Then $e_2$ must be mapped to $v_{12}v'_4$, and consequently $f_2$ must be mapped to $v'_1v'_2$. Then the only vertex between $v_{12}$ and $v'_1$ is the isolated vertex $v$ so we cannot map $e_1$ to any edge.
    
    In all cases we arrived to a contradiction. Thus we are left to show that adding an arbitrary edge incident to $v$ in $H$ creates a copy of $H$.
    
    Connect $v$ and another vertex $w$ with a new edge. There are four different cases, and for each we can find a copy of $G$ in $H$ plus the new edge, finishing the proof.
    
    Case 1: $w < v_{10}$. 
    The following map gives a copy of $G$: $e_1,e_2,e_3$ are mapped to $wv,v_{12}v'_4$ and $v'_3v'_5$, respectively, while $f_1$ and $f_2$ are mapped to $v_{10}v_{11}$ and $v'_1v'_2$, respectively.
    
    Case 2: $w\in\{v_{10},v_{11}\}$. 
    The following map gives a copy of $G$: $e_1,e_2,e_3$ are mapped to  
    $v_3v_7,v_6v_{12}$ and $wv$, respectively, while $f_1$ and $f_2$ are mapped to $v_{4}v_{5}$ and $v_8v_9$, respectively.
    
    Case 3: $w\in\{v_{12}, v'_1\}$. 
    The following map gives a copy of $G$: $e_1,e_2,e_3$ are mapped to $v_1v_{11},v_4v'_{3}$ and $v'_2v'_5$, respectively, while $f_1$ and $f_2$ are mapped to $v_{2}v_{3}$ and $vw$, respectively.    
    
    Case 4: $w > v'_1$. 
    The following map gives a copy of $G$: $e_1,e_2,e_3$ are mapped to $v_3v_10,v_7v'_{1}$ and $vw$, respectively, while $f_1$ and $f_2$ are mapped to $v_{4}v_{5}$ and $v_{11}v_{12}$, respectively.
\end{proof}

\begin{figure}[h]
\centering
\includegraphics[width=0.7\textwidth]{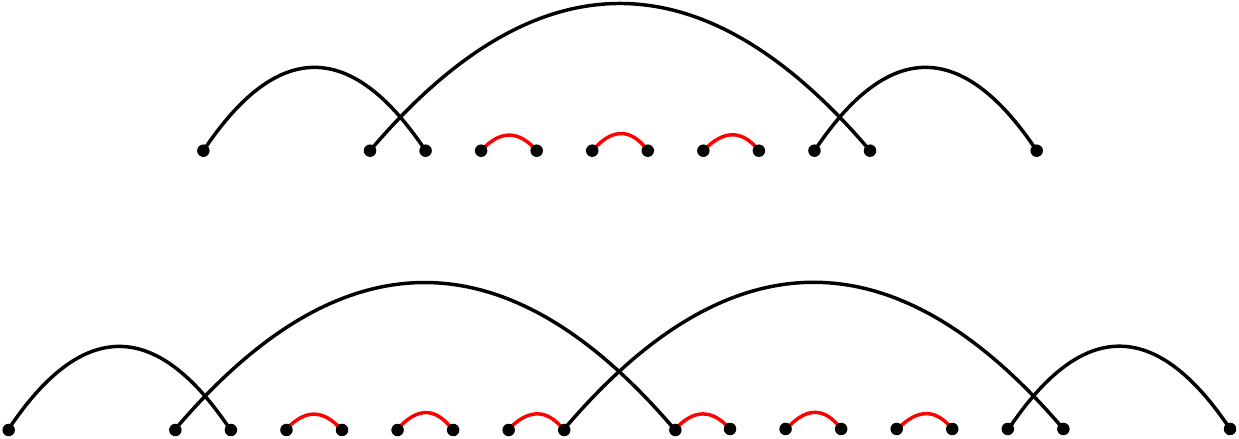}
\caption{\label{fig:030} Graph $\Gamma_{\{0,3,0\}}$ and $\HH{(\Gamma_{\{0,3,0\}})}$.}
\end{figure}

We will use the following construction to prove for a large class of graphs that they have bounded saturation function:

\begin{defi}\label{def:hg}
    Let $G$ be a graph on $n$ vertices. Assume the first and last vertex of $G$ both are incident to an isolated edge. Let $G'$ (resp. $G'')$ be the graph on $n-2$ vertices we get from $G$ by deleting the endvertices of the edge incident to the first (resp. last) vertex. Denote by $\HH(G)$ the graph on $2n-6$ vertices we get by placing a copy of $G'$ on the first $n-2$ vertices and a copy of $G''$ on the last $n-2$ vertices. Note that these two copies overlap on $2$ vertices. See Figure \ref{fig:030}.  
\end{defi}

\begin{defi}
    Let $A,G_1,G_2$ be three graphs on $n_0,n_1,n_2$ vertices, respectively. 
   We denote by $G_1 \mathbb{y} A\mathbb{x} G_2$ the family of graphs on $n_1+n_0+n_2-2$ vertices that can be obtained by choosing arbitrarily vertices $i,j$ such that both $i$ and $j$ are at least $n_1$ and at most $n_1+n_0-1$ and placing a copy of $G_1$ on the vertex set $\{1,\dots,n_1-1\}\cup\{i\}$, a copy of $G_2$ on the vertex set $\{j\}\cup\{n_1+n_0,\dots,n_1+n_0+n_2-2\}$ and a copy of $A$ on the vertex set $\{n_1,\dots,n_1+n_0-1\}$. See Figure \ref{fig:10G01}.
\end{defi}

\begin{figure}[h]
	\centering
	\includegraphics[width=0.6\textwidth]{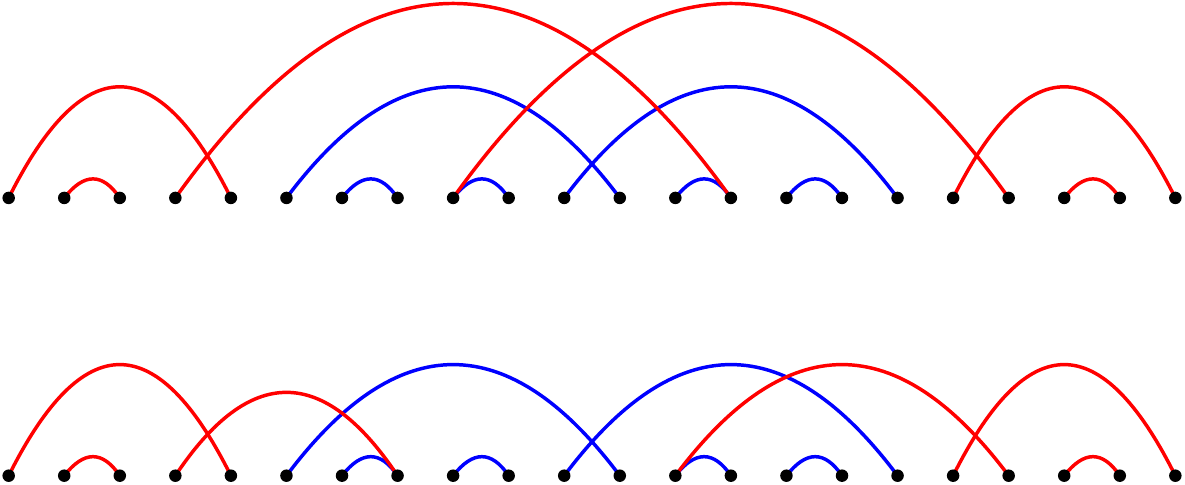}
	\caption{\label{fig:10G01} Two possible members of a family $\Gamma_{\{1,0\}} \mathbb{y} \Gamma_{\{2,2\}} \mathbb{x} \Gamma_{\{0,1\}}.$}
\end{figure}

\begin{lem}\label{lemma:g3ag3}
     Let $A$ be an ordered graph and $G\in\Gamma_3 \mathbb{y} A \mathbb{x} \Gamma_3$. If $G$ is not separable and has no degree zero or degree two vertices, then $\HH(G)$ avoids $G$.
\end{lem}

\begin{proof}
   Assume on the contrary that there is a $G_1$ copy of $G$ in $\HH(G)$. See Figure \ref{fig:3A3} for an illustration where certain edges of $\HH(G)$ are labeled.
   
   Assume first that neither $e_2$ nor $e_3$ is contained in $G_1$. As $\HH(G) \setminus \{e_2,e_3\}$ is separable but $G$ is not, $G$ has to be fully contained in the subgraph induced by either the first half or the second half of the vertices. However, both subgraphs have less than $|V(G)|$ vertices, a contradiction.
   
   Assume second that at least one of $e_2$ and $e_3$, wlog. $e_2$, is contained in $G_1$. In this case, as $G$ is a graph with no degree-two vertices, $e_4$ cannot be in $G_1$. As $\HH(G) \setminus \{e_4\}$ is separable but $G$ is not, $G_1$ has to be fully contained in the subgraph induced by the set of vertices either preceding and including $r(e_3)$ or succeeding $r(e_3)$. However, both subgraphs have less than $|V(G)|$ vertices, a contradiction.     
\end{proof}

\begin{figure}[h]
\centering
\includegraphics[width=0.75\textwidth]{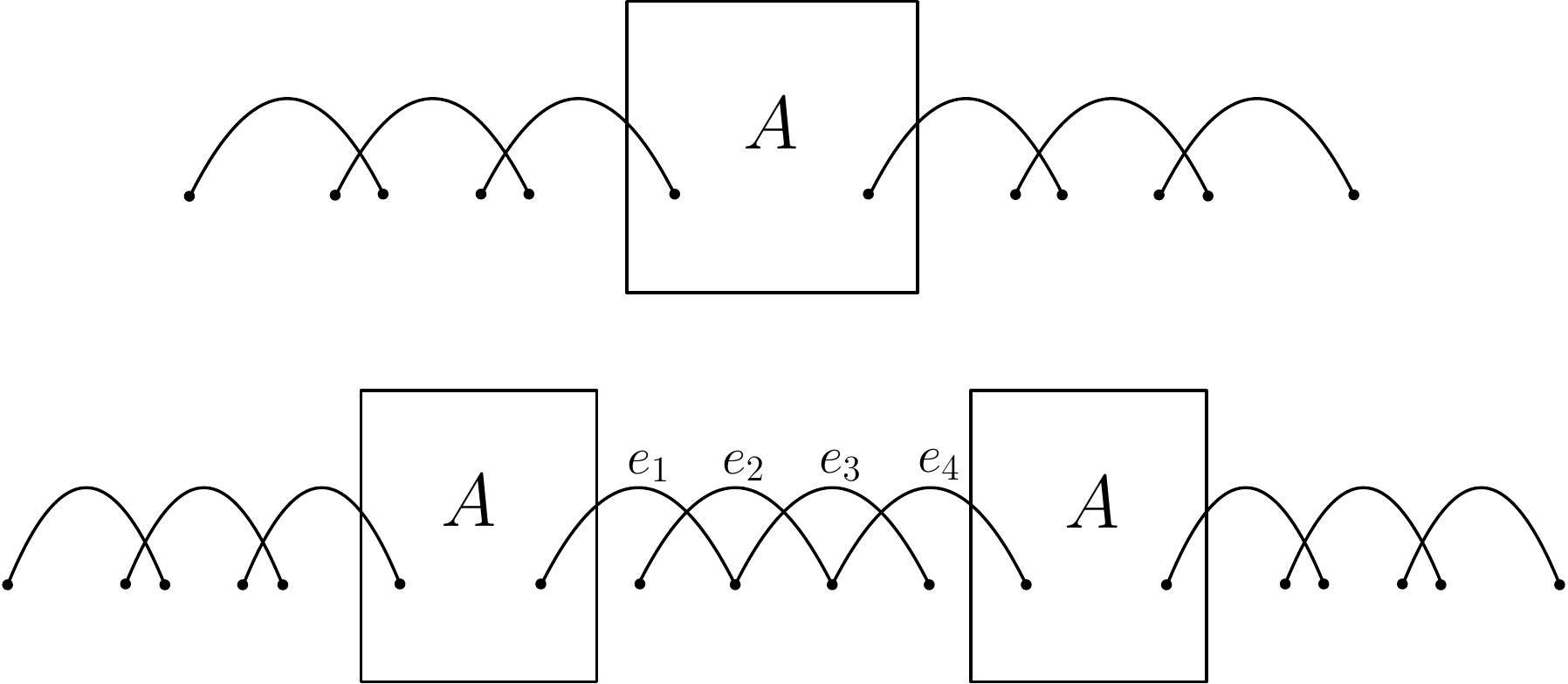}
\caption{\label{fig:3A3} Graph $G \in \Gamma_3 \mathbb{y} A \mathbb{x} \Gamma_3$ and $\HH{(G)}$.}
\end{figure}

\begin{thm}
    Let $A$ be an ordered graph and $$G\in\Gamma_{\{0,0,1,0\}} \mathbb{y} A \mathbb{x} \Gamma_{\{0,1,0,0\}} \text{ or } G\in\Gamma_{\{0,0,1,0\}} \mathbb{y} A \mathbb{x} \Gamma_{\{0,0,0,0\}}.$$ If $G$ is not separable and has no degree zero or degree two vertices, then $sat_{<}(n,G)=O(1)$.
\end{thm}

\begin{proof}
	We first prove the case when $G\in\Gamma_{\{0,0,1,0\}} \mathbb{y} A \mathbb{x} \Gamma_{\{0,1,0,0\}}$.
	
	Let $H$ be the graph on Figure \ref{fig:4A4} bottom, with some of the edges labeled. We prove that $H$ is a witness for $G$ which implies  $sat_{<}(n,G)=O(1)$. 
	
	First we show that $H$ avoids $G$. Assume on the contrary that $H$ contains a copy $G_1$ of $G$. Notice first that $\Gamma_{\{0,0,1,0\}} \mathbb{y} A \mathbb{x} \Gamma_{\{0,1,0,0\}}$ is a subfamily of $\Gamma_3 \mathbb{y} A' \mathbb{x} \Gamma_3$ for an appropriate choice of $A'$ and that we get $H$ from $\HH{(G)}$ by adding the edge set  $\{a,b,c\}$ and the isolated vertices $v$ and $v'$. Lemma \ref{lemma:g3ag3} implies that $\HH{(G)}$ avoids $G$ thus $G_1$ must contain at least one edge from the edge set $\{a,b,c\}$.
    
    Case 1: $b$ is not in $G_1$ and $a$ is in $G_1$.
    As $G$ has no vertex of degree two, $e_2$ cannot be in $G_1$.
    
    Case 1.1: $e_1$ is in $G_1$. In this case $e_3$ cannot be in $G_1$ since $G$ has no vertex of degree two. Thus $e_2$ and $e_3$ are not in $G_1$ and using that $G$ is not separable, $G_1$ is completely to the left from $v$, but there are less than $|V(G)|$ such vertices, a contradiction. 
    
    Case 1.2: $e_1$ is not in $G_1$. In this case $b,e_1,e_2$ are not in $G_1$ and using that $G$ is not separable, either every vertex of $G_1$ is completely to the left from $l(e_3)$ or it is completely in the rest of the vertices, but there are less than $|V(G)|$ vertices on each side, a contradiction.
    
    Case 2: $b$ is not in $G_1$ and $c$ is in $G_1$. Notice that in Case 1 we did not use where exactly the left endvertex of $a$ and right endvertex of $c$ lies, thus a symmetrical argument works also in Case 2.
    
    \begin{figure}[h]
    \centering
    \includegraphics[width=0.8\textwidth]{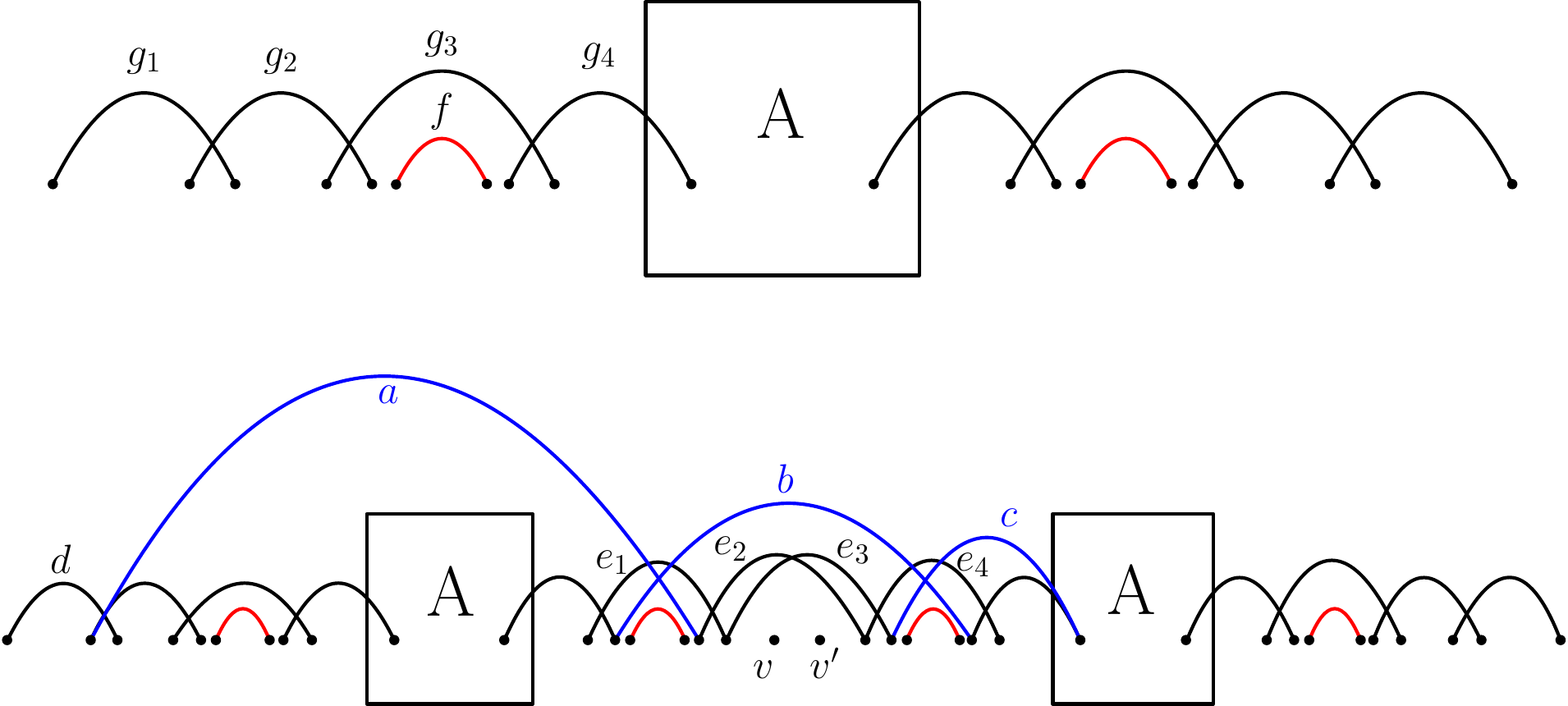}
    \caption{\label{fig:4A4} Graph $G \in \Gamma_{\{0,0,1,0\}} \mathbb{y} A \mathbb{x}  \Gamma_{\{0,1,0,0\}}$ and its witness graph $H$.}
    \end{figure}

    Case 3: $b$ is in $G_1$. Then the other two edges incident to the endvertices of $b$ cannot be in $G_1$ as $G$ has no vertex of degree two.

    Case 3.1: $a$ and $c$ are not in $G_1$. Then using that $G$ is not separable, $G_1$ has to be in the closed interval defined by $l(e_1)$ and $r(e_4)$. However, there are less than $|V(G)|$ vertices in this interval (as $|v \in V(H): l(e_1) \leq v \leq r(e_4)| = 13 < 19 = |V(\Gamma_{\{0,0,1,0\}})|+|V(\Gamma_{\{0,1,0,0\}})|-1 \leq V(G)$), a contradiction.
    
    Case 3.2: $a$ is in $G_1$. Since $l(a)$ is the second smallest vertex in $H$,  only $g_1$ or $g_2$ can be mapped to $a$, in which case $b$ must play the role of $g_2$ or $g_3$, respectively. Since the number of vertices in $H$ greater than $r(b)$ is $|V(G)|-8$ and the number of vertices in $G$ greater than $r(g_2)$ is $|V(G)|-5$, it follows that $b$ cannot play the role of $g_2$. Thus $g_2$ is mapped to $a$ and $g_3$ is mapped to $b$. In this case $g_4$ is mapped to $e_4$ or $c$. In both cases there is no room to map the endvertices of $f$, a contradiction.
    
    Case 3.3: $a$ is not in $G_1$ and $c$ is in $G_1$. As $G$ is not separable, $G_1$ lies on the right side of (and including) the vertex $l(e_1)$. As $l(b)$ is the second vertex on this side, $b$ must play the role of $g_1$ or $g_2$ in $G_1$. Similar to Case 3.2, since the number of vertices in $H$ greater than $r(b)$ is $|V(G)|-8$ and the number of vertices in $G$ greater than $r(g_2)$ is $|V(G)|-5$, this is also impossible.

    It remains to show that connecting $v$ with an arbitrary vertex $w \in V(H)$ creates a copy of $G$. Let $G'$ and $G''$ be the subgraphs whose union is $\HH(G)$ as defined in Definition \ref{def:hg}. First, if $w<l(e_3)$ then adding $wv$ to $H\setminus\{a,b,c\}$ creates a copy of $G$ in which the first edge of $G$, $g_1$, is mapped to $wv$ and the rest is mapped to $G''$. Similarly, if $w>r(e_2)$ then adding $vw$ to $H\setminus\{a,b,c\}$ creates a copy of $G$ in which the last edge of $G$ is mapped to $vw$ and the rest is mapped to $G'$. Finally, if $w\in\{l(e_3),v',r(e_2)\}$ then the following mapping gives a copy of $G$ in $H$: $f$ is mapped to the edge connecting $v$ and $w$, further, $g_1,g_2,g_3,g_4$ are mapped to $d, a, b, c$, respectively, and the remaining edges are mapped to a subgraph of $G''$ in the same way as in the case $w<l(e_3)$. 
    
   	The case when $G\in\Gamma_{\{0,0,1,0\}} \mathbb{y} A \mathbb{x} \Gamma_{\{0,0,0,0\}}$ (note that $\Gamma_{\{0,0,0,0\}}=\Gamma_4$) can be proved practically verbatim, except that the witness is the graph on Figure \ref{fig:4A4} bottom minus the second (the one covered by $e_1$) and fourth (the one right to the second copy of $A$) red minedge.    
\end{proof}

\begin{cor}
    Let $A$ be an ordered graph and $$M\in\Gamma_{\{0,0,1,0\}} \mathbb{y} A \mathbb{x} \Gamma_{\{0,1,0,0\}} \text{ or } M\in\Gamma_{\{0,0,1,0\}} \mathbb{y} A \mathbb{x} \Gamma_{\{0,0,0,0\}}.$$ If $M$ is an ordered matching and it is not separable, then $sat_{<}(n,G)=O(1)$.
\end{cor}

\begin{cor}
    $$sat_{<}(n,\Gamma_{\{0,0,1,m_1,m_2,...,m_k,1,0,0\}})=O(1) \text{ and } sat_{<}(n,\Gamma_{\{0,0,1,m_1,m_2,...,m_k,0,0,0\}})=O(1),$$ for $k \geq 2$ and $m_i \geq 0$ for all $1 \leq i \leq k$.
\end{cor}

\section{Saturation of cyclically ordered graphs}\label{sec:cyclic}
	
In this section we consider cyclically ordered graphs. We prove dichotomy for their saturation function as well along with infinitely many examples for both cases (bounded and linear).

Given a graph $C$ on a cyclically ordered vertex set, for vertices $u,v,x$ we write $u<x<v$ if starting with $u$ in clockwise direction we first meet $x$ then $v$. We define the open interval $I_{u,v} = \{x \in V(C): u < x < v\}$. From now on every graph we consider is cyclically ordered even if we don't say it explicitly.

\begin{thm}\label{thm:cyclicdich}
	Given a cyclically ordered graph $C$, we either have $sat_{\circlearrowright}(n,C)=O(1)$ or $sat_{\circlearrowright}(n,C)=\Theta(n)$. 
\end{thm}
\begin{proof}
The first part of the proof is almost identical to that of Theorem \ref{thm:ordereddich}. Let $H_n$ be the host graph saturating $C$ on $n$ vertices with $sat_{\ca}(n,G)$ edges. If $C$ has no isolated vertices then if there exists an $n_0$ such that $H_{n_0}$ contains two adjacent isolated vertices, then we can multiply these vertices to get host graphs of size $n>n_0$ with the same number of edges, showing that $sat_{\ca}(n,C)\le sat_{\ca}(n_0,C)=O(1)$ for $n\ge n_0$. If $C$ contains isolated vertices then instead of two we require $|V(C)|$ many consecutive isolated vertices and get to the same conclusion.

Thus either $sat_{\ca}(n,G)=O(1)$ or there are no two (resp. $|V(C)|$ many) consecutive isolated vertices in the host graphs, which implies $sat_{\ca}(n,C)=\Omega(n)$.

We are left to prove that $sat_{\ca}(n,C)=O(n)$ always holds. Assume that $C$ has $k$ vertices and let $s$ be the minimum length of an interval $I$ of the vertices of $C$ in the cyclic order such that the vertices of the interval hit every edge of $C$. 

Let $H$ be the cyclically ordered graph on $n$ vertices in which an interval $J$, $|J|=s-1$, of the vertices is connected to every other vertex and there are no other edges besides these in $H$. This $H$ avoids $C$ by the definition of $s$.  Further, $H$ has at most $(s-1)n=O(n)$ edges. Now we greedily add edges to $H$ until we get a graph $H'$ that saturates the property of avoiding $C$. We claim that in $H'$ every vertex not in $J$ has degree at most $2k-s-3$, which implies that $H'$ has $O(n)$ edges, as required.

Assume on the contrary that there is a vertex $v$ outside $J$ with degree at least $2k-s-2$. The vertices of $H$ outside $J\cup\{v\}$ form two intervals. In at least one of these intervals $v$ has at least $\lceil(2k-s-2-(s-1))/2\rceil=k-s$ neighbors. Now on the vertices of $J\cup\{v\}$ and these $k-s$ vertices there is a copy of $G$ where the vertices of $J\cup\{v\}$ play the role of the interval $I$, a contradiction.
\end{proof}

We note that similarly to ordered graphs, when $sat_{\ca}(n,G)=O(1)$ then there exists a number $n_0$ such that $sat_{\ca}(n,G)=sat_{\ca}(n_0,G)$ for $n\ge n_0$.

We define minedges and superedges for cyclically ordered graphs:

\begin{defi}
    Let $C$ be a cyclically ordered graph, and let $uv \in E(C)$. If either $I_{u,v}$ or $I_{v,u}$ is empty and the degree of $u$ and $v$ is one, then $uv$ is a \emph{minedge}.
\end{defi}

\begin{defi}
	Let $C$ be a cyclically ordered graph, and let $uv \in E(C)$. If both of $I_{u,v}$ and $I_{v,u}$ induce an edge, then $uv$ is a \emph{bisuperedge}.
	
\end{defi}

The proofs of the following two statements are analogous to the proofs of Observation \ref{obs:single} and Claim \ref{claim:minedge} and are left to the reader:

\begin{obs}\label{obs:singlecyclic}
	If $G$ is a single edge then $sat_{\ca}(n,G)= 0$.
\end{obs}

\begin{claim}\label{claim:minedgecyclic}
	If $C$ contains no minedge then $sat_{\circlearrowright}(n,G)= \Theta(n)$.
\end{claim}
	
We now show infinite classes of graphs that have a bounded saturation function.

\begin{defi}
	Let $L_k$ be the cyclically ordered matching on vertex set $[2k]$ and edge set $\{(2i-1)(2i):i=1,\dots,k\}$. \footnote{Notice that we have already defined $L_k$ on linearly ordered vertices, here we extend this linear order to a cyclic ordering. Hopefully this ambiguity will not lead to confusion.}
	
	Let $X_k$ be the cyclically ordered matching on vertex set $[n]$ with $n=2k+4$ vertices and edge set $\{(1)(n-1),(2)(n)\}$ plus a copy of $L_k$ placed on the vertex set $\{3,4,\dots, n-2\}$. See Figure \ref{fig:Xk} for an illustration.	
\end{defi}

\begin{lem}\label{lem:l3cyclic}
    Let $C$ be a cyclically ordered graph that contains $L_3$ as a subgraph such that $sat_{\circlearrowright}(n,C) = O(1)$. Let $H_n$ be a host graph on $n$ vertices saturating $C$. Then there exists $n_0$ s.t. every $H_n$ ($n\ge n_0$) contains an isolated vertex $v$ and another vertex $w$ such that adding the edge $vw$ to $H$ we get a copy of $C$ in which $vw$ plays the role of an edge of $C$ which is not a minedge of $C$.
\end{lem}

\begin{proof}
	Assume on the contrary that such a pair $v,w$ does not exist. If $n_0$ is big enough then any host graph $H_n$ ($n\ge n_0$) saturating $C$ must contain at least two isolated vertices. Let $v$ be an isolated vertex of $H_n$.	First let $w$ be the first (in clockwise order) isolated vertex after $v$. Adding the edge $vw$ to $H$ a copy of $C$ is created, which we denote by $C'$. $C'$ contains a copy of $L_3$ which we denote by $L_3'$. No matter if $vw$ is in $L_3'$ or not, there are always two minedges, $m_1m_2$ and $m_3m_4$, in $L_3'$, such that starting from $v$ and going clockwise the vertices come in the following order: $m_1,m_2,m_3,m_4$. Now let $C''$ be a copy of $C$ created when adding the edge $vm_3$ to $C$. By the indirect assumption, $vm_3$ has to be a minedge in $C''$. This implies that the rest of the vertices of $C''$ can be found either in $I_{v,m_3}$ or $I_{m_3, v}$. In both cases we can replace $vm_3$ by some edge (either $m_3m_4$ or $m_1m_2$) to get a copy of $C$ already in $H$, a contradiction.
\end{proof}

\begin{claim} \label{claim:lkcyclic}
	$sat_{\circlearrowright}(n,L_k) = \Theta(n)$ for $k \geq 2$.
\end{claim}

\begin{proof}
    First we deal with the case $k=2$, and assume on the contrary that $sat_{\circlearrowright}(n,L_2) = O(1)$. Thus there exists a host graph $H$ saturating $L_2$ with isolated vertex $v$. As $H$ contains at least one edge, we can choose an edge $xy \in E(H)$ such that $I_{v,y}$ does not contain any edges. By connecting $v$ and $y$, by assumption we obtain a copy of $L_2$ in $H$, which means that either $I_{v,y}$ or $I_{y,v}$ contains an edge. The first case is impossible by our choice of $xy$. In the second case this edge and $xy$ together form a copy of $L_2$ in $H$, a contradiction.
    
    Assume that $sat_{\circlearrowright}(n,L_k) = O(1)$ for some $k \geq 3$. Trivially, $L_k$ contains $L_3$, and we can apply Lemma \ref{lem:l3cyclic}. Thus there exists a host graph $H$ saturating $L_k$ with isolated vertex $v$ and another vertex $w$ such that by adding the edge $vw$ to $H$ we get a copy of $L_k$ in which $vw$ plays the role of an edge of $L_k$ which is not a minedge of $L_k$. However, every edge of $L_k$ has to be a minedge, thus this is a contradiction.
\end{proof}

We are ready to prove a statement similar to Theorem \ref{thm:minsup}:

\begin{thm}\label{thm:minsupcyclic}
    If the cyclically ordered graph $C$ contains $L_3$ and every edge of $C$ is a minedge or a bisuperedge then $sat_{\ca}(n,C)=\Theta(n)$.
\end{thm}

\begin{proof}	
	Let $M$ be the set of minedges and $B$ be the set of bisuperedges of $C$. 
	If $B$ is empty then we are done by Claim \ref{claim:lkcyclic} thus we can assume that $B$ is not empty. Assume on the contrary that $sat_{\circlearrowright}(n,C)= O(1)$. 
    
	Let $H_n$ be a host graph saturating $C$ with $sat_{\circlearrowright}(n,C)$ edges. If $n$ is big enough, Lemma \ref{lem:l3cyclic} guarantees the existence of vertices $v,w$ in $H_n$. Wlog. assume that $w$ is the first (in clockwise order) after $v$ with the property guaranteed by Lemma \ref{lem:l3cyclic}. Then $vw$ is not a minedge in the copy $C'$ of $C$ created when adding $vw$ to $H_n$. As every edge of $C$ is either a minedge or a bisuperege, $vw$ must be a bisuperedge and thus there is an edge $e$ in $H$ on the vertices of the interval $I_{vw}$. Let $w'$ be the counterclockwise neighbor of $w$ (notice that it cannot be $v$). By our choice of $w$, in the copy $C''$ created when adding $vw'$ to $H$, $vw'$ must be a minedge. Replacing $vw'$ in $C''$ with the edge $e$ we get a copy of $C$ in $H$, a contradiction.
\end{proof}

\begin{figure}[h]
\centering
\includegraphics[width=0.7\textwidth]{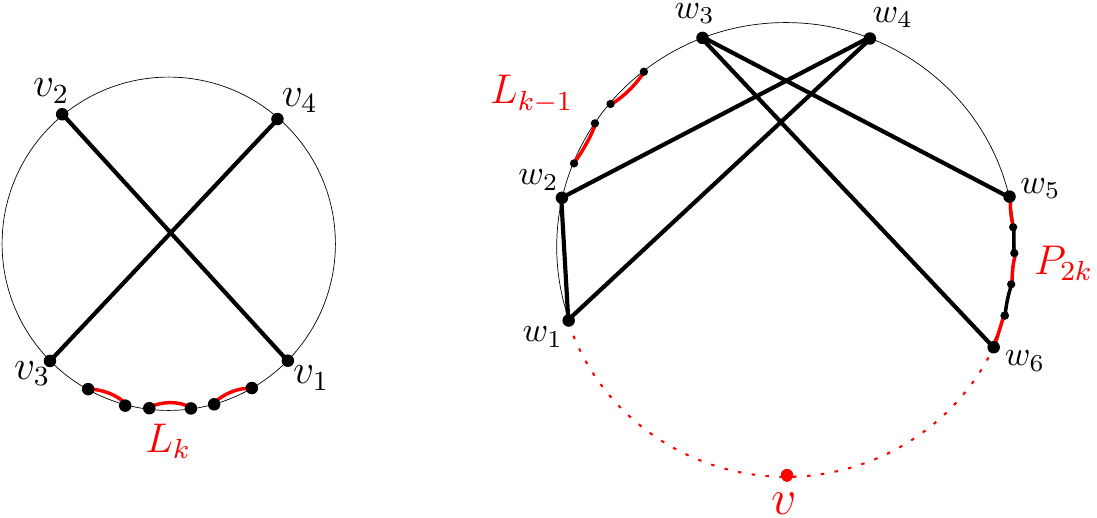}
\caption{\label{fig:Xk} Graph $X_k$ and the graph $H_n$.}
\end{figure}

Finally, we give an infinite class of cyclically ordered graphs with bounded saturation function (note that $sat_{\circlearrowright}(n,X_0)=\Theta(n)$ by Claim \ref{claim:minedgecyclic}):

\begin{thm}\label{thm:xk}
    $sat_{\circlearrowright}(n,X_k) = O(1)$ for every $k\geq 1$.
\end{thm}

\begin{proof}
   We call a pair of edges $v_1v_2$ and $v_3v_4$ \emph{crossing} if $v_1<v_3<v_2$ and $v_2<v_4<v_1$. 
   
   Let $H_n$ be the graph as drawn on Figure \ref{fig:Xk}, with some vertices and edge sets labeled. It has $3k+3=O(1)$ edges. 
   
   First we show that $H=H_n$ avoids $X_k$. Assume on the contrary and let $e,f$ be the crossing pair of edges of $H$ playing the role of $v_1v_2$ and $v_3v_4$. Notice that in $H$ there are only $4$ crossing pairs of edges. Also, both $e$ and $f$ must have $2k+1$ non-isolated edges on one of their sides (as this holds for $v_1v_2$ and $v_3v_4$ in $X_k$). This does not hold for $w_1w_4$ and $w_3w_6$. Thus we must have $\{e,f\}=\{w_2w_4,w_3w_5\}$. The endvertices of these two edges split the vertices into four intervals and a copy of $L_k$ is contained in one of them. First, $I_{w_3,w_4}$ and $I_{w_4,w_5}$ are empty. Moreover, $w_2w_3$ contains only $L_{k-1}$. Finally, $I_{w_5,w_2}$ contains exactly $2k$ vertices which is exactly as many as $L_k$ has. However, the vertex $w_1$ is not connected with any other vertex from the interval which is not the case with $L_k$. Contradiction.
   
   Next we show that adding any edge incident to the isolated vertices in $I_{w_6w_1}$ creates a copy of $X_k$. As there are only constant many vertices outside $I_{w_6w_1}$, this implies that greedily adding $O(1)$ edges to $H_n$ we get a graph saturating $G$ and having $O(1)$ edges, finishing the proof.\footnote{We could define witness graphs for cyclically ordered graphs similarly to how we have defined them for ordered graphs. With this terminology, $H$ would be a witness for $X_k$.}
   
	We denote by $M_k$ (resp. $M'_{k-1}$) the set of odd (resp. even) edges of the edges of the path $P_{2k}$ on $2k$ vertices of $H$. Let $v$ be an arbitrary isolated vertex in $I_{w_6w_1}$ and let $w$ be another arbitrary vertex. We have to check that adding the edge $vw$ to $H$ creates a copy of $X_k$. There are three cases:
    
    Case 1: $w_6 \leq w \leq w_1$. In this case $vw$ together with $M'_{k-1}$ forms $L_k$, $w_2w_4$ and $w_5w_3$ play the role of the crossing edges in the copy of $X_k$.
    
    Case 2: $w_1 < w < w_4$. $M_k$ forms $L_k$, $vw$ and $w_1w_4$ play the role of the crossing edges in the copy of $X_k$.
    
    Case 3: $w_4 \leq w < w_6$.  $L_{k-1}$ and $w_1w_2$ form $L_k$, $vw$ and $w_3w_6$ play the role of the crossing edges in the copy of $X_k$.
\end{proof}

\section{Semisaturation}\label{sec:ssat}

In this section we consider the semisaturation problem for (cyclically) ordered graphs. 

A graph $H$ is semisaturating $G$ if adding any edge to $H$ creates a new copy of $G$. Let $ssat_<(n,G)$ (resp. $ssat_\ca(n,G)$) be the minimum size of a $G$-semisaturated ordered (resp. cyclically ordered) graph on $n$ vertices.

As it was the case with $0$-$1$ matrices, the semisaturation problem turns out to be much easier than the saturation problem. We are able to characterize exactly which (cyclically) ordered graphs have bounded semisaturation function.

\begin{thm} \label{thm:semiord}
    For an ordered graph $G$, $ssat_{<}(n,G)=O(1)$ if and only if all the following hold:
    \begin{enumerate}
        \item $G$ contains a minedge,
        \item $G$ contains an edge connecting the first vertex with a degree one vertex,
        \item $G$ contains an edge connecting the last vertex with a degree one vertex,
    \end{enumerate}
    
    and $ssat_{<}(n,G)=\Theta(n)$ otherwise.
\end{thm}

\begin{proof}
	Notice that $ssat_{<}(n,G)\le sat_{<}(n,G)=O(n)$ by Theorem \ref{thm:ordereddich}.
	
    Assume first that at least one of the conditions is not satisfied. First, if $G$ has no minedge then an $H$ semisaturating $G$ cannot have two consecutive isolated vertices, as otherwise connecting two such vertices could not create a new copy of $G$. This implies that $ssat_{<}(n,G) \geq n/4$. If the second (resp. third) condition is not satisfied then there is no isolated vertex in $H$, possibly besides the first (resp. last) vertex, as otherwise connecting the first (resp. last) vertex with an isolated vertex could not create a new copy of $G$. From this we get $ssat_{<}(n,G) \geq (n-1)/2$. 
    
    We assume now that there is a minedge $m_1m_2$, an edge $u_1u_2$ connecting the first vertex, $u_1$, with a degree one vertex $u_2$ and an edge $v_1v_2$ connecting the last vertex, $v_2$, with a degree one vertex $v_1$. For simplicity we regard the vertex set of $G$ to be $[k]$ and thus $m_1$, etc. refer not only to vertices but also they are positive integers (in particular $u_1=1,v_2=k$ and $m_2=m_1+1$).
    For $n$ big enough we construct a graph $H=H_n$ semisaturating $G$. We denote the set of the first $\max(u_2+m_1-3,v_1-1)$ vertices of $H$ by  $B_1$, and the set of the last $\max(2k-v_1-m_2-1,k-u_2)$ vertices by $B_2$. We connect all the vertices in $B_1$ (resp. $B_2$), also we add all edges between $B_1$ and $B_2$ to $H$ and no other edges. $H$ clearly has $O(1)$ edges, as required. We are left to prove that adding an edge to $H$ creates a new copy of $G$. Notice that the only edges that we are able to add must have at least one vertex outside $B_1$ and $B_2$. Thus we connect such an isolated vertex $v$ with an arbitrary vertex $w$.    
    First, if $w$ is among the first $m_1-1$ vertices of $H$ then there are at least $u_2-2$ vertices in $B_1$ between $w$ and $v$. Also, there are at least $k-u_2$ vertices after $v$ in $B_2$. Thus a new copy of $G$ is created such that $wv$ plays the role of $u_1u_2$. Similarly, if $w$ is among the last $k-m_2$ vertices of $H$ then another copy of $G$ is created such that $vw$ plays the role of $v_1v_2$. Finally, in the remaining case a copy of $G$ is created such that the edge connecting $v$ and $w$ plays the role of the minedge $m_1m_2$. 
\end{proof}

\begin{thm}
    For a cyclically ordered graph $C$, $ssat_{\circlearrowright}(n,C)=O(1)$ if and only if $C$ contains a minedge, and $ssat_{\circlearrowright}(n,C)=\Theta(n)$ otherwise.
\end{thm}

\begin{proof}
	Notice that $ssat_{\ca}(n,G)\le sat_{\ca}(n,G)=O(n)$ by Theorem \ref{thm:cyclicdich}.
		
    First, if $C$ has no minedge then an $H$ semisaturating $C$ cannot have two consecutive isolated vertices, as otherwise connecting two such vertices could not create a new copy of $G$. This implies that $ssat_{\ca}(n,G) \geq n/4$.
          
    Now, assume that $C$ has $k$ vertices and it contains a minedge. For $n$ big enough we construct a graph $H=H_n$ semisaturating $G$. Take an interval $J$ of the vertices of $H$ of size $2k-4$ and connect all the vertices of $J$ with each other. $H$ has no other edges, i.e., the rest of the vertices are isolated. $H$ clearly has $O(1)$ edges, as required. We are left to prove that adding an edge to $H$ creates a new copy of $G$.
    
    A new edge must connect an isolated vertex $v$ and a vertex $w$ in $H$. As either $I_{v,w}\cap J$ or $I_{w,v}\cap J$ has at least $k-2$ vertices, adding the edge $vw$ creates a copy of $C$ in $H$ in which $vw$ plays the role of a minedge.
\end{proof}

\section{Discussion}\label{sec:discussion}

It would be interesting to determine the order of magnitude of the saturation function for further linked matchings. Perhaps it is not out of reach to fully characterize the linked matchings with bounded saturation function, or at least those which are of the form $\Gamma_{\{m_1,m_2,\dots,m_k\}}$ with $m_i\in\{0,1\}$.

Notice that to a linearly ordered graph naturally corresponds a cyclically ordered graph by extending the linear order of the vertices to a cyclic order (and similarly, for a bipartite ordered graph corresponds an ordered and in turn a cyclically ordered graph). In particular, in the cyclic case one can also regard linked matchings. Considering linked matchings in the cyclic setting did not fit in the scope of this paper, nevertheless, it would be nice to see how differently they behave compared to the ordered case. To this end, it is not hard to check that the witnesses for $\Gamma_{\{0,1,0\}}$ and $\Gamma_{\{1,0,1\}}$ from Section \ref{sec:linked} also avoid these graphs cyclically and thus $sat_{\ca}(n,\Gamma_{\{0,1,0\}})=O(1)$ and $sat_{\ca}(n,\Gamma_{\{1,0,1\}})=O(1)$. Yet for example the witness we had for $\Gamma_{\{1,1,0\}}$ contains a copy of $\Gamma_{\{1,1,0\}}$ in the cyclic setting, so for this graph we do not know the answer in the cyclic case.

Let us now compare the three ordered variants (ordered bipartite, ordered, cyclically ordered):
\begin{enumerate}
	\item Ordered bipartite vs. ordered. Corollary \ref{cor:intchrom} states that for every ordered bipartite graph $G$ we have $sat_<(n,G)=\Theta(n)$, irrespective of what $sat_{\01}(n,G)$ is.
	\item Ordered bipartite vs. cyclic. First, $sat_{\01}(n,X_0)=\Theta(n)$ and $ sat_{\ca}(n,X_0)=\Theta(n)$. Second, $sat_{\01}(n,X_1)=\Theta(n)$ and $ sat_{\ca}(n,X_1)=O(1)$. Third, if $sat_{\01}(n,G)=O(1)$ then $sat_{\ca}(n,G)=\Theta(n)$. Indeed, in this case $G$ is not decomposable (\cite{01sat}, see the Introduction) and so in particular $G$ cannot contain a minedge when the vertices are regarded cyclically and then by Claim \ref{claim:minedgecyclic} we have $sat_{\ca}(n,G)=\Theta(n)$.
	\item Ordered vs. cyclic. First, $sat_{<}(n,X_0)=\Theta(n)$ and $ sat_{\ca}(n,X_0)=\Theta(n)$. Second, $sat_{<}(n,X_1)=\Theta(n)$ and $ sat_{\ca}(n,X_1)=O(1)$. Third, $sat_{<}(n,\Gamma_{\{0,1,0}\})=O(1)$ and $ sat_{\ca}(n,\Gamma_{\{0,1,0\}})=O(1)$. Finally, while we cannot exclude the possibility, we have no example where the ordered saturation is bounded and the cyclically ordered saturation is linear.
\end{enumerate}

The missing case in the ordered vs. cyclic case can be phrased this way:
\begin{prob}\label{prob:ordvscyc}
	Is it true for every ordered graph $G$ that $sat_{\ca}(n,G)=O(sat_{<}(n,G))$?
\end{prob}

A graph showing that the answer is false would be a graph with bounded $sat_{<}$ and linear $sat_{\ca}$. The difficulty lies partially in the fact that we do not know that many types of graphs belonging to any of these two classes.

Notice that $sat_{\ca}(n,G)=sat_<(n,\{G_1,G_2,...,G_k\})$ where $G_i$ is the graph we get by `cutting' the cyclic order of the vertices after the $i$th vertex to get a linear order on the vertices. We can get the graph $G_i$ from $G_1$ by shifting the linear order on the vertices by $i-1$. Thus the above problem can be phrased the following way: is it true that forbidding every shifted version of a graph $G_1$ cannot increase the order of magnitude of the saturation function? 

We did not regard the case when multiple graphs are forbidden. Most probably many of our results generalize to this case, yet in general it would be interesting to see if forbidding multiple graphs can exhibit new behaviours. In particular the following problem (for which the answer may easily be false) is a good first step also to solve Problem \ref{prob:ordvscyc}:

\begin{prob}\label{prob:orderedvscyclic}
	Is it true for every pair of ordered graphs $A,B$ that $sat_{<}(n,\{A,B\})=O(sat_{<}(n,A))$?
\end{prob}

This is true if we replace the saturation function with the extremal function (even for multiple forbidden graphs), in all settings (unordered graphs and the three settings of ordered graphs). Further, while in the case of graphs it is possible that the extremal function of a collection of graphs is actually the minimum of the extremal functions of the members of the collection (this is a well-known open problem of extremal graph theory), when considering $0$-$1$ matrices, a collection of matrices can have an extremal function strictly smaller than the extremal function of any of the members of the collection \cite{tardos_2019}.
For a brief treatment of the saturation problem for a collection of matrices see the Conclusion of \cite{berendsohn2} and for an application to saturation problems on the Boolean poset see \cite{satgrid}.

We don't know if an algorithm exists that always stops that decides for a given (cyclically) ordered graph if its saturation function is bounded. For a more detailed discussion about this computational problem see the Discussion of \cite{01sat}, where it is considered for $0$-$1$ matrices.

Throughout this paper we were interested in vertex-ordered graphs, while we did not consider the case of edge-ordered graphs. These graphs were regarded in the context of extremal problems \cite{gerbner2020turn,tardos_2019}, however, nothing is yet known about their saturation function.

\bigskip
\noindent \textbf{Acknowledgement}

\bigskip

This work was initiated during a visit of the first author at Eötvös Loránd University throughout the Erasmus Programme.
We thank D\"om\"ot\"or P\'alv\"olgyi for encouraging this collaboration and telling us about this problem. Also, for his help in the initial stages of this research, in particular concerning the dichotomy result of ordered graphs and that nested graphs have linear saturation function.

\bibliographystyle{plainurl}
\bibliography{saturation}

\begin{thebibliography}{10}

\bibitem{berendsohn1}
Benjamin~Aram Berendsohn.
\newblock Matrix patterns with bounded saturation function, 2020.
\newblock \href {http://arxiv.org/abs/2012.14717} {\path{arXiv:2012.14717}}.

\bibitem{berendsohn2}
Benjamin~Aram Berendsohn.
\newblock An exact characterization of saturation for permutation matrices,
  2021.
\newblock \href {http://arxiv.org/abs/2105.02210} {\path{arXiv:2105.02210}}.

\bibitem{Brass2003}
Peter Brass, Gyula K{\'a}rolyi, and Pavel Valtr.
\newblock {\em A Tur{\'a}n-type Extremal Theory of Convex Geometric Graphs},
  pages 275--300.
\newblock Springer Berlin Heidelberg, Berlin, Heidelberg, 2003.
\newblock \href {https://doi.org/10.1007/978-3-642-55566-4_12}
  {\path{doi:10.1007/978-3-642-55566-4_12}}.

\bibitem{brualdicao}
Richard~A. Brualdi and Lei Cao.
\newblock Pattern-avoiding (0,1)-matrices and bases of permutation matrices.
\newblock {\em Discrete Applied Mathematics}, 304:196--211, 2021.

\bibitem{conlon2016ordered}
David Conlon, Jacob Fox, Choongbum Lee, and Benny Sudakov.
\newblock Ordered ramsey numbers, 2016.
\newblock \href {http://arxiv.org/abs/1410.5292} {\path{arXiv:1410.5292}}.

\bibitem{graphsatsurvey}
Bryan~L. Currie, Jill~R. Faudree, Ralph~J. Faudree, and John~R. Schmittm.
\newblock A survey of minimum saturated graphs.
\newblock {\em The Electronic Journal of Combinatorics}, DS19, 2021.
\newblock \href {https://doi.org/https://doi.org/10.37236/41}
  {\path{doi:https://doi.org/10.37236/41}}.

\bibitem{2020saturation}
G\'abor Dam\'asdi, Bal{\'a}zs Keszegh, David Malec, Casey Tompkins, Zhiyu Wang,
  and Oscar Zamora.
\newblock Saturation problems in the ramsey theory of graphs, posets and point
  sets.
\newblock {\em European Journal of Combinatorics}, 95:103321, 2021.

\bibitem{dudek2013minimum}
Andrzej Dudek, Oleg Pikhurko, and Andrew Thomason.
\newblock On minimum saturated matrices.
\newblock {\em Graphs and Combinatorics}, 29(5):1269--1286, 2013.

\bibitem{erdoshajnalmoon}
Paul Erd\H{o}s, Andr{\'a}s Hajnal, and John~W. Moon.
\newblock A problem in graph theory.
\newblock {\em The American Mathematical Monthly}, 71(10):1107--1110, 1964.

\bibitem{ferrara2017saturation}
Michael Ferrara, Bill Kay, Lucas Kramer, Ryan~R Martin, Benjamin Reiniger,
  Heather~C Smith, and Eric Sullivan.
\newblock The saturation number of induced subposets of the boolean lattice.
\newblock {\em Discrete Mathematics}, 340(10):2479--2487, 2017.

\bibitem{frankl2020vc}
N{\'o}ra Frankl, Sergei Kiselev, Andrey Kupavskii, and Bal{\'a}zs Patk{\'o}s.
\newblock {VC-saturated set systems}.
\newblock {\em arXiv preprint arXiv:2005.12545}, 2020.

\bibitem{01sat}
Radoslav Fulek and Bal{\'{a}}zs Keszegh.
\newblock Saturation problems about forbidden 0-1 submatrices.
\newblock {\em {SIAM} J. Discret. Math.}, 35(3):1964--1977, 2021.
\newblock \href {https://doi.org/10.1137/20M1376327}
  {\path{doi:10.1137/20M1376327}}.

\bibitem{furedikim}
Zolt\'an F\"uredi and Younjin Kim.
\newblock Cycle-saturated graphs with minimum number of edges.
\newblock {\em Journal of Graph Theory}, 73(2):203--215, 2013.

\bibitem{geneson}
Jesse Geneson.
\newblock Almost all permutation matrices have bounded saturation functions.
\newblock {\em Electron. J. Comb.}, 28(2):P2.16, 2021.

\bibitem{Gerbner2013}
D{\'a}niel Gerbner, Bal{\'a}zs Keszegh, Nathan Lemons, Cory Palmer,
  D{\"o}m{\"o}t{\"o}r P{\'a}lv{\"o}lgyi, and Bal{\'a}zs Patk{\'o}s.
\newblock Saturating sperner families.
\newblock {\em Graphs and Combinatorics}, 29(5):1355--1364, 2013.

\bibitem{gerbner2020turn}
D\'aniel Gerbner, Abhishek Methuku, D\'aniel~T. Nagy, D\"om\"ot\"or
  P\'alv\"olgyi, G\'abor Tardos, and M\'at\'e Vizer.
\newblock Tur\'an problems for edge-ordered graphs, 2020.
\newblock \href {http://arxiv.org/abs/2001.00849} {\path{arXiv:2001.00849}}.

\bibitem{satgrid}
D\'aniel Gerbner, D\'aniel~T. Nagy, Bal\'azs Patk\'os, and M\'at\'e Vizer.
\newblock Forbidden subposet problems in the grid.
\newblock {\em Discrete Mathematics}, 345(3):112720, 2022.

\bibitem{kaszonyituza}
L\'aszl\'o K\'aszonyi and Zsolt Tuza.
\newblock Saturated graphs with minimal number of edges.
\newblock {\em Journal of Graph Theory}, 10(2):203--210, 1986.

\bibitem{keszegh2020induced}
Bal{\'a}zs Keszegh, Nathan Lemons, Ryan~R. Martin, D{\"o}m{\"o}t{\"o}r
  P{\'a}lv{\"o}lgyi, and Bal{\'a}zs Patk{\'o}s.
\newblock Induced and non-induced poset saturation problems.
\newblock {\em Journal of Combinatorial Theory, Series A}, 184:105497, 2021.

\bibitem{Pach2006}
J{\'a}nos Pach and G{\'a}bor Tardos.
\newblock Forbidden paths and cycles in ordered graphs and matrices.
\newblock {\em Israel Journal of Mathematics}, 155(1):359--380, Dec 2006.
\newblock \href {https://doi.org/10.1007/BF02773960}
  {\path{doi:10.1007/BF02773960}}.

\bibitem{dompc}
D\"om\"ot\"or P\'alv\"olgyi.
\newblock Personal communication, 2021.

\bibitem{tardos_2019}
G\'abor Tardos.
\newblock {\em Extremal theory of vertex or edge ordered graphs}, pages
  221--236.
\newblock London Mathematical Society Lecture Note Series. Cambridge University
  Press, 2019.

\end{thebibliography}
\end{document}